\documentclass[12pt]{amsart}

\usepackage{amsmath} 
\usepackage{amssymb} 
\usepackage{amsthm}
\usepackage{amscd}
\newtheorem{thm}{Theorem}[section]
\newtheorem{lem}[thm]{Lemma}
\newtheorem{lem-dfn}[thm]{Lemma-Definition}
\newtheorem{prop}[thm]{Proposition}
\newtheorem{cor}[thm]{Corollary}

\theoremstyle{definition}
\newtheorem{defn}[thm]{Definition}
\newtheorem{exam}[thm]{Example}

\newtheorem{quest}[thm]{Question}

\newtheorem{prob}[thm]{Problem}

\newtheorem*{acknowledgement}{Acknowledgement}
\theoremstyle{remark}
\numberwithin{equation}{section}
\DeclareMathOperator{\Supp}{Supp}
\DeclareMathOperator{\Spec}{Spec}
\DeclareMathOperator{\spec}{Spec}

\DeclareMathOperator{\Min}{Min}

\DeclareMathOperator{\Ext}{Ext}
\DeclareMathOperator{\Hom}{Hom}

\DeclareMathOperator{\tr}{trace}
\DeclareMathOperator{\res}{res}

\DeclareMathOperator{\CMtype}{CMtype}
\newcommand{\m}{\mathfrak m}

\newcommand{\fra}{{\mathfrak a}}

\newcommand{\frp}{{\mathfrak p}}
\newcommand{\frq}{{\mathfrak q}}
\begin{document}
%%%%%%%%%%%%%%%%%%%%%%%%%%%%%%%%%%%%%%%%%%%%%%%%%%%%%%%%%%%%%%% 
%%    Title 
%%%%%%% %%%%%%% %%%%%%% %%%%%%% %%%%%%% %%%%%%% %%%%%%% 
\title{Ulrich ideals on rational triple points of dimension two}
%%%%%%%%%%%%%%%%%%%%%%%%%%%%%%%%%%%%%%%%%%%%%%%%%%%%%%%%%%%%%%% 
%%%%%%%%%%%%%%%%%%%%%%%%%%%%%%%%%%%%%%%%%%%%%%%%%%%%%%%%%%%%%%% 
%%   Information for first author
%%%%%%% %%%%%%% %%%%%%% %%%%%%% %%%%%%% %%%%%%% %%%%%%% 
\author{Kyosuke Maeda}
\address[Kyosuke Maeda]{Department of Mathematics, 
College of Humanities and Sciences, 
Nihon University, Setagaya-ku, Tokyo, 156-8550, Japan}
\email{maeda127k@gmail.com}
%%%%%%%%%%%%%%%%%%%%%%%%%%%%%%%%%%%%%%%%%%%%%%%%%%%%%%%%%%%%%%% 
%%    Information for second author
\author{Ken-ichi Yoshida}
\address[Ken-ichi Yoshida]{Department of Mathematics, 
College of Humanities and Sciences, 
Nihon University, Setagaya-ku, Tokyo, 156-8550, Japan}
\email{yoshida.kennichi@nihon-u.ac.jp}
\thanks{The second author was partially supported 
by JSPS Grant-in-Aid for Scientific Research (C) 
Grant Numbers, 24K06678}
%%%%%%%%%%%%%%%%%%%%%%%%%%%%%%%%%%%%%%%%%%%%%%%%%%%%%%%%%%%%%%% 
%%%%%%%%%%%%%%%%%%%%%%%%%%%%%%%%%%%%%%%%%%%%%%%%%%%%%%%%%%%%%%% 
%    General info
%%%%%%% %%%%%%% %%%%%%% %%%%%%% %%%%%%% %%%%%%% %%%%%%% 
% 13B22 Integral closure of rings and ideals [See also 13A35]; integrally closed rings, related rings (Japanese, etc.)
% 13A30 Rees rings ...
% 13H10 Special types (Cohen-Macaulay, Gorenstein, Buchsbaum, etc.)
% 14B05 Singularities
\date{\today}
%\keywords{good ideal, $p_g$-ideal, Rees algebra, normal %Hilbert coefficient, rational singularity}
%%% Yoshida271012
%\dedicatory{}
%%%%%%%%%%%%%%%%%%%%%%%%%%%%%%%%%%%%%%%%%%%%%%%%%%%%%%%%%%%%%%% 
%%%%%%%%%%%%%%%%%%%%%%%%%%%%%%%%%%%%%%%%%%%%%%%%%%%%%%%%%%%%%%% 
%              Abstract 
%%%%%%% %%%%%%% %%%%%%% %%%%%%% %%%%%%% %%%%%%% %%%%%%% 
\begin{abstract}
In this paper, we prove the canonical trace ideal $\tr_A(\omega_A)$ 
is an Ulrich ideal for any two-dimensional rational triple point $A$. 
Using this, we classify all Ulrich ideals on rational triple points. 
\par
Moreover, we show that if $(A,\m)$ is a two-dimensional quotient singularity with the multiplicity $e \ge 4$  
then $\m$ is the unique Ulrich ideal of $A$.   
As a result, we can classify all Ulrich ideals of $A$ if $A$ is  
either a rational triple point or a quotient singularity. 
\end{abstract}
%%%%%%%%%%%%%%%%%%%%%%%%%%%%%%%%%%%%%%%%%%%%%%%%%%%%%%%%%%%%%%% 
\maketitle
%%%%%%%%%%%%%%%%%%%%%%%%%%%%%%%%%%%%%%%%%%%%%%%%%%%%%%%%%%%%%%%
%%%%%%%%%%%%%%%%%%%%%%%%%%%%%
\section{Introduction}

Let $(A,\m)$ be a Cohen-Macaulay local ring.  
An $\m$-primary ideal $I \subset A$ is called an \textit{Ulrich ideal}
if $I^2=QI$ holds for some minimal reduction $Q$ of $I$ and $I/I^2$ is a free $A/I$-module; see \cite{GOTWY1}. 
Let ${\mathcal X}_A$ denote the set of all Ulrich ideals of $A$. 

\par 
Goto et al.~\cite{GOTWY1} introduced the notion of Ulrich ideals and 
gave their several properties. 
In \cite{GOTWY2}, they studied Ulrich ideals on two-dimensional 
rational singularities $A$ and prove that any Ulrich ideal $I \subset A$ is an integrally closed ideal represented on the minimal resolution; 
see \cite[Theorem 6.2]{GOTWY2}.  
Moreover, in \cite[Theorem 1.4]{GOTWY2}, they classified all Ulrich ideals 
in a Gorenstein rational  singularity (that is, a rational singularity with multiplicity $2$) using McKay correspondence, 
and gave a criterion for Ulrich ideals of non-Gorenstein rational 
singularities in terms of geometric words; see \cite[Theorem 6.4]{GOTWY2}.  
\par \vspace{2mm}
The main purpose of this paper is to classify all Ulrich ideals of 
two-dimensional rational  triple points (i.e. rational singularities with  multiplicity $3$) and quotient singularities.  

\par \vspace{2mm}
Let us explain the organization of the paper. 
In Section 2, we classify all Ulrich ideals of rational  triple points $A$. 
Let $\tr_A(\omega_A)$ be the canonical trace ideal of $A$  and put $\res(A)=\ell(A/\tr_A(\omega_A))$, 
the residue of $A$; see 
Definition \ref{Def-Trace} for details. 
One of the main results in this paper is the following. 

%%%  Theorem 1.1
\begin{thm}[\textrm{See Section 3}]  \label{Main-RTP} 
Let $(A,\m)$ be a two-dimensional rational  triple point. 
Then there exist  a minimal system of generators $x_1,\ldots, x_n$ of $\m$ and an integer $c \ge 0$ such that
\begin{enumerate}
\item $\tr_A(\omega_A)=(x_1,\ldots,x_{n-1},x_n^{c+1})$ is an Ulrich ideal. 
\item ${\mathcal X}_A=\{J \,|\, J \supset \tr_A(\omega_A)\}$. 
In particular, $\res(A)=\sharp({\mathcal X}_A)=c+1$.
\item  $A$ is nearly Gorenstein $($see Definition $\ref{NG-Gor}$$)$ if and only if $\sharp ({\mathcal X}_A)=1$. 
\end{enumerate}
\end{thm}

\par \vspace{1mm}
In order to prove the theorem, we use a classification theorem for rational triple point by Artin \cite{Ar} and Tyurina \cite{Tyu} and 
calculate the canonical trace ideal $\tr_A(\omega_A)$. 
The key point of our argument is any rational triple point has 
$\CMtype=2$. 
The complete list of Ulrich ideals will be given in Section 3.

\par \vspace{1mm}
In Section 4, we study Ulrich ideals for quotient singularities using 
\cite[Theorem 6.4]{GOTWY2} and Brieskorn \cite{Br}. 
The main result in this section is the following. 

%  Theorem 1.2
\begin{thm}[\textrm{See Theorem \ref{UQuotMult4}, Corollary 
\ref{QuotUlrich}}] \label{UQ}
If $A$ is a two-dimensional quotient singularity, then 
$\sharp ({\mathcal X}_A) \le 2$.   
\par 
Especially, if the multiplicity $e_0(A) \ge 4$, then 
${\mathcal X}_A=\{\m\}$. 
\end{thm}

\par \vspace{1mm}
Ding \cite{Ding} classified all two-dimensional 
nearly Gorenstein quotient singularities. 
Two-dimensional nearly Gorenstein quotient singularities have 
the unique Ulrich ideal $\m$. 
But the converse is not true in general. 
In fact, there exists a quotient singularity $A$ such that 
${\mathcal X}_A=\{\m\}$ but $A$ is not nearly Gorenstein; 
see Example \ref{Non-NG}. 
\par \vspace{1mm}
On the other hand, we can classify all Ulrich ideals on two-dimensional quotient singularities explicitly.

%%%%%%%%%%%%%%%%%%%%%%%%%%%%%%%%%%%%%%%%%%%%%%%%%%%%
%%%%%%%%%%%%%%%%%%%%%%%%%%%%%%%%%%%%%%%%%%%%%%%%%%%
%%%%%%%%%%%  Section 2
\medskip
\section{Preliminaries}

Let $(A,\m)$ be a Cohen-Macaulay local ring and suppose 
that $k=A/\m$ is an infinite field.  
For an $A$-module $M$, $\ell(M)$ (resp. $\mu(M)$) denotes  
the length (resp. the number of minimal system of generators) of $M$. 
\par 
Let $I \subset A$ be an $\m$-primary ideal. 
Then there exists a parameter ideal $Q \subset I$ such that 
$I^{n+1}=QI^n$ for some $n \ge 1$. 
Then $Q$ is called a \textit{minimal reduction} of $I$. 
The \textit{multiplicity} of $I$ (denoted by $e_0(I)$) is given 
by $e_0(I)=\ell_A(A/Q)$. 
In particular, $e_0(A)$ is called the \textit{multiplicity} of the ring $A$.   
\par 
In general, it is well-known that $\mu(\m) \le e_0(\m)+\dim A-1$ holds, and $A$ is said to 
have \textit{minimal multiplicity} if equality holds. 

%%%%%%%%%%%%%%%%%%%%%%%%%%%%%%%%%%%%%%%
\subsection{Ulrich ideals}
First we recall the definition of Ulrich ideals and good ideals; see \cite{GIW, GOTWY1}.

%%% Definition
\begin{defn}[\textbf{Good ideal, Ulrich ideal}]
An $\m$-primary ideal $I \subset A$ is called an \textit{Ulrich ideal} (resp.  a \textit{good ideal})
if $I^2=QI$ holds for some minimal reduction $Q$ of $I$ and $I/I^2$ is a free $A/I$-module (resp. $I=Q\colon I$).
\end{defn}
\par \vspace{1mm}
\par 
We denote the set of all Ulrich ideals in $A$ by ${\mathcal X}_A$. 
One has the following criterion for $I$ to be an Ulrich ideal. 

%%% Lemma
\begin{lem}[\textrm{cf. \cite[Lemma 2.3]{GOTWY1}}] \label{CriterionUlrich}
Let $I \subset A$ be an $\m$-primary ideal. 
Then the following conditions are equivalent. 
\begin{enumerate}
\item $I$ is an Ulrich ideal. 
\item $I^2=QI$ holds for some minimal reduction $Q$ of $I$ and $e_0(I)=(\mu(I)-\dim A+1) \cdot \ell(A/I)$. 
\end{enumerate} 
\end{lem}	

\par \vspace{1mm}
Any Ulrich ideal is a good ideal (see \cite[Corollary 2.6]{GOTWY1}). 
For instance, if $A$ is a two-dimensional Cohen-Macaulay local ring of minimal multiplicity,  then 
the maximal ideal $\m$ is an Ulrich ideal and thus a good ideal.  
Moreover, any power $\m^n$ $(n \ge 2)$ 
is also a good ideal but not an Ulrich ideal.  

%%% Problem
\begin{prob}
Determine the set  ${\mathcal X}_A$. 
\end{prob}

%%%%%%%%%%%%%%%%%%%%%%%%%%%%%%%%%%%%%%%%%%%%%%%%
\medskip 
\subsection{Canonical trace ideal} 

Let $(A,\m)$ be a Cohen-Macaulay local ring of $d=\dim A \ge 1$, 
and suppose $A$ is a homomorphic image of a regular local ring $S$
of $v=\dim S$. 
Then $A$ admits a canonical module $\omega_A \cong \Ext_S^{v-d}(A,S)$. 
Note that if $A$ is a generically Gorenstein ring, that is,  
$A_{\frp}$ is Gorenstein for every $\frp \in \Min(A)$, then 
$\omega_A$ is isomorphic to an ideal of $A$; e.g. \cite[Korollar 6.7]{HK}. 

%%%  Definition
\begin{defn}[\textbf{Trace ideal}; see \cite{HHS}] \label{Def-Trace}
For an $A$-module $M$, the \textit{trace ideal} of $M$ is defined to be 
\[
\tr_A(M)=\sum_{f \in \Hom_A(M,A)} f(M). 
\] 
\par 
Now suppose that $A$ is generically Gorenstein. 
Then $\tr_A(\omega_A)$ is an ideal, which 
is called the \textit{canonical trace ideal} of $A$. 
\par
Moreover, we put $\res(A)=\ell(A/\tr_A(\omega_A))$, the \textit{residue} of $A$. 
\end{defn}

\par 
If $\tr_A(\omega_A) =A$, then $A$ is Gorentein. 
Recall the notion of nearly Gorensteinness. 

%%% Definition
\begin{defn}[\textrm{cf. Herzog et. al \cite{HHS}}] \label{NG-Gor}
A Cohen-Macaulay ring 
$A$ is called \textit{nearly Gorenstein} if $\tr_A(\omega_A) \supset \m$. 
\end{defn}

%%% Question
\begin{quest}
When is $A$ nearly Gorenstein? 
\end{quest}

\par 
Now let $S$ be a regular local ring and $\fra \subset S$ an 
prime ideal. Put $A=S/\fra$.   
Suppose that $A$ is a non-Gorenstein, Cohen-Macaulay 
local domain with $e_0(A)=3$ (e.g. a rational triple point). 
Then $A$ has minimal multiplicity and hence $\CMtype(A)=2$.   
Hence Hilbert-Burch's theorem yields the following 
minimal free resolution of $A$ over $S$:  
\[
0 \to S^2 \stackrel{\mathbb{M}}{\longrightarrow} S^3 \to S \to A \to 0, \quad \text{where}\; \fra=I_2(\mathbb{M}).  
\]
\cite[Corollary 3.4]{HHS} implies the following lemma, which plays a 
key role in the next section.  

%%%  Lemma 
\begin{lem} \label{Burch}
Suppose that $A$ is a non-Gorenstein, Cohen-Macaulay 
local domain with $e_0(A)=3$. 
Then we have  
$\tr_A(\omega_A)=I_1(\mathbb{M})A$. 
\end{lem}

\par 
For example, if $A=k[[x,y,z,t]]/I_2
{\small \begin{pmatrix}
x & t^{n+1} &t^{n+1}+z\\[1mm]
t^{n+1} & y & z
\end{pmatrix}}$, where $n \ge 0$ is an integer, 
then 
\[
\tr_A(\omega_A)=(x,t^{n+1}, t^{n+1}+z, t^{n+1}, y, z)=(x,y,z,t^{n+1}). 
\]

%%%%%%%%%%%%%%%%%%%%%%%%%%%%%%
\medskip
\subsection{Rational singularity}
~\par \vspace{1mm}
Let $(A,\m)$ be a two-dimensional excellent normal local domain 
containing an algebraically closed field $k$ of characteristic $0$. 
Suppose that $A$ is not regular,  
and let $f\:X \to \spec A$ be a resolution of singularities 
with exceptional divisor $E:=f^{-1}(\m)$ unless otherwise specified. 
Let $E=\bigcup_{i=1}^rE_i$ be the decomposition into  irreducible 
components of $E$; see \cite{OWY1}.  
\par 
The weighted dual graph $\Gamma$ of $E=\bigcup_{i=1}^rE_i$ as follows: 
The set of vertices of $\Gamma$ consists of prime divisors $E_i$. 
We connect the vertices $E_i$ and $E_j$ if $E_i$ and $E_j$ intersect. 
Moreover, we assign the weight $-a$ to the vertex corresponding to 
$E_i$ if $E_i^2=-a$.

\par \vspace{2mm}
For any $\m$-primary integrally closed ideal $I$, there exists a 
resolution of singularities $X \to \Spec A$ and an anti-nef cycle 
$Z$ on $X$ such that  the ideal sheaf 
$I\mathcal{O}_X=\mathcal{O}_X(-Z)$ is invertible and 
$I=H^0(X, I\mathcal{O}_X)$.
Then $I$ is said to be \textit{represented on} $X$ by $Z$ and 
it is denoted by $I=I_Z$. 

\par
The invariant  $p_g(A):=\ell(H^1(X,\mathcal{O}_X))$ 
is called the \textit{geometric genus} of $A$. 
This is independent on the choice of the resolution of singularities $X \to \Spec A$. 

\medskip
%%%  Definition
\begin{defn}[\textbf{Rational singularity}] \label{Defn-Rational}
A ring $A$ is called a \textit{rational singularity} if $p_g(A)=0$. 
\end{defn}

\par \vspace{1mm}
%Let $e=e_0(A)$ be the multiplicity of $A$ with respect to $\m$.    
Let $A$ be a rational singularity. 
Then $A$ has minimal multiplicity and $\CMtype(A)=e-1$, where 
$e=e_0(A)$ denotes the multiplicity of $A$. 
Then $A$ is Gorenstein if and only if $e=2$. 
Such a ring $A$ is called a \textit{rational double point} 
(abbr. RDP). 

\par \vspace{1mm} 
It is known that a rational double point $(A,\m)$ is  isomorphic to one of the following hypersurfaces:
\par \vspace{1mm}
$(A_n)$\;
$k[[x,y,z]]/(z^2+x^2+y^{n+1})\;(n\ge 1)$, 

$(D_n)$\;
$k[[x,y,z]]/(z^2+x^2y+y^{n-1})\;(n\ge 4)$, 

$(E_6)$\;
$k[[x,y,z]]/(z^2+x^3+y^4)$, 

$(E_7)$\;
$k[[x,y,z]]/(z^2+x^3+xy^3)$, 

$(E_8)$\;
$k[[x,y,z]]/(z^2+x^3+y^5)$. 

\par \vspace{3mm}
All Ulrich ideals for rational double points are completely 
classified as follows
(see \cite[Theorem 1.4]{GOTWY2}):  
\par \vspace{1mm}
$(A_{2m})$\;
$\{(x, y, z), (x, y^2, z),\ldots , (x, y^m, z)\}$, 
\par \vspace{1mm}
$(A_{2m+1})$\;
$\{(x, y, z), (x, y^2, z),\ldots , (x, y^{m+1}, z)\}$, 
\par \vspace{1mm}
$(D_{2m})$\;
$\{(x, y, z), (x, y^2, z),\ldots , (x, y^{m-1}, z),$ 
\par \qquad \qquad $(x+\sqrt{-1}y^{m-1}, y^m, z), (x-\sqrt{-1}y^{m-1}, y^m, z), (x^2, y, z)\}$, 
 \par \vspace{1mm}
$(D_{2m+1})$\;
$\{(x, y, z), (x, y^2, z),\ldots , (x, y^m, z), (x^2, y, z)\}$, 
\par \vspace{1mm}
$(E_6)$\;
$\{(x, y, z), (x, y^2, z)\}$, 
\par \vspace{1mm}
$(E_7)$\;
$\{(x, y, z), (x, y^2, z), (x, y^3, z)\}$, 
\par \vspace{1mm}
$(E_8)$\;
$\{(x, y, z), (x, y^2, z)\}$. 

\par \vspace{3mm}
In what follows, let $(A,\m)$ be a rational singularity with $e \ge 3$. 
Suppose that $X \to \Spec A$ is the minimal resolution of singularities, that is, any resolution of singularities factors through 
$X \to \Spec A$. 
This means that $E=\cup_{i=1}^n E_i$ does not contain any $(-1)$-curve. 
\par 
A cycle $Z = \sum_{i=1}^n a_iE_i$ is positive if $a_i \ge 0$ for each $i$ 
and $a_j \ne 0$ for some $j$. 
A positive cycle $Z$ is called  \textit{anti-nef} if $ZE_i < 0$ for every $i$. 
Let $Z_0$ be the fundamental cycle on $X$, 
that is, $Z_0$ is 
the minimum element of anti-nef cycles on $X$. 
\par 
Note that the maximal ideal $\m$ is represented on the minimal 
resolution of singularities by the fundamental cycle $Z_0$. 
Let  $K_X$ denote the canonical divisor of $X$. 
Then for any positive cycle $Y$ on $X$, we put
\[
p_a(Y)=\dfrac{Y^2+K_XY}{2}+1. 
\]
Then $p_a(Y)$ is called the \textit{arithmetic genus} of $Y$. 
Since $E_i \cong \mathbb{P}^1$, we have $p_a(E_i)=0$. 
Thus $K_XE_i=-E_i^2-2$ for every $i$. 
See e.g. \cite{Ish} for more details. 

\par 
In order to prove the above theorem, we recall the structure theorem 
of Ulrich ideals; see \cite[Theorem 6.4]{GOTWY2}. 
\par 
Let $I \subset A$ be an Ulrich ideal which is not a maximal ideal $\m$.  
Since any Ulrich ideal is an integrally closed ideal represented on 
the minimal resolution, $I=I_Z$ for some anti-nef cycle $Z$ on $X$. 
Then \cite[Theorem 6.4]{GOTWY2} shows that there exists an integer $s \ge 1$  and anti-nef cycles $Z_1,\ldots,Z_s$ and positive cycles 
$0 < Y_s \le \cdots \le Y_1 (\le Z_0)$ such that 
\begin{eqnarray*}
Z  =  
Z_s &=& Z_{s-1}+Y_s \\ 
&\vdots & \\
Z_2 & = & Z_1+Y_2 \\
Z_1 & = & Z_0+Y_1  
\end{eqnarray*}
and 
\[
Y_kZ_{k-1}=p_a(Y_k)= K_X(Z_0-Y_k)=0 \quad \text{for each $k=1,2,\ldots,s$.}
\]
In particular, 
\begin{equation} \label{Eq:UlrichSupp}
\{E_i \,|\, -E_i^2=b_i \ge 3\} \subset \Supp(Y_1)
\subset \{E_i \,|\, E_iZ_0 =0\}. 
\end{equation}

\par \vspace{1mm}
So we have the following criterion for having the unique Ulrich ideal $\m$, 
which plays a key role in Section 4. 

\begin{lem} \label{KeyLemma}
If there exists a curve $E_j$ so that $-E_j^2=b_j \ge 3$ and 
$-E_jZ_0 > 0$, then ${\mathcal X}_A=\{\m\}$. 
\end{lem}

\begin{proof}
If there exist an Ulrich ideal $I \ne \m$, then one can take 
a positive cycle $Y_1 \le Z_0$ which satisfies 
Eq.(\ref{Eq:UlrichSupp}). 
But this contradicts the assumption. 
\end{proof}

\bigskip 
We next show the relationship between the canonical trace ideal and 
Ulrich ideals; see \cite[Proposition 2.12, Corollary 2.14]{GK}.

% Proposition 
\begin{prop} \label{Ulrich-Can}
Suppose that $A$ is a rational singularity which is not Gorenstein.    
Then $I \supset \tr_A(\omega_A)$ holds for any Ulrich ideal $I \subset A$.  
\end{prop} 

\begin{proof}
Suppose that $I$ is an Ulrich ideal of $A$. 
Since $A$ is a rational singularity with $\CMtype(A) \ge 2$, 
we have $\mu(I) =\CMtype(A)+2 \ge 4$ by 
\cite[Corollary 2.12]{GTT} and $A$ is a generically Gorenstein ring. 
Thus \cite[Proposition 2.12]{GK} implies $I \supset \tr_A(\omega_A)$. 
\end{proof}
 
\par 
So it is natural to ask the following question. 
\begin{quest}
When is $\tr_A(\omega_A)$ an Ulrich ideal?  
\end{quest}

%%%%%%%%%%%%%%%%%%%%%%%%%%%%%%%%%%%%%%%%%%%%%%%%%%%%
%%%%%%%%%%%%%%%%%%%%%%%%%%%%%%%%%%%%%%%%%%%%%%%%%%%%
%%%%%%%%%%% Section 3 
\medskip 
\section{Rational Triple Points}

\par \vspace{3mm}
A rational singularity with $e_0(A)=3$ is called a \textit{rational triple point}
(abbr. RTP). 
Artin \cite{Ar} and Tyurina \cite{Tyu} classified all rational triple points: 
For a dual graph $H$ (of the minimal resolution), we denote the corresponding coordinate ring by $R(H)$:

%%%%%%%%%%%%%%%%%%%%%%%% Almn %%%%%%%%%%%%%%%%%%%%%%%%%%%%%%%%%%
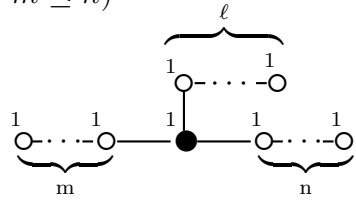
\begin{enumerate}

\item $A_{\ell,m,n}(0\le \ell \le m\le n)$

\begin{picture}(200,60)(-40,0)
    \thicklines
\put(25,18){{\tiny $1$}}
\put(30,12){\circle{6}}
\put(33,12){\line(1,0){4}}
\put(28,7){$\underbrace{\phantom{AAAA}}_{\text {m}}$}
\put(38,9){$\cdots$}
\put(52,12){\line(1,0){4}}
  \put(55,18){{\tiny $1$}}
\put(61,12){\circle{6}}
\put(65,12){\line(1,0){20}}
  \put(83,18){{\tiny $1$}}
\put(91,12){\circle*{8}}
\put(90,16){\line(0,1){15}}
\put(95,12){\line(1,0){20}}
  \put(115,18){{\tiny $1$}}
\put(120,12){\circle{6}}
  \put(145,18){{\tiny $1$}}
\put(150,12){\circle{6}}
\put(128,9){$\cdots$}
\put(123,12){\line(1,0){4}}
\put(142,12){\line(1,0){4}}
\put(118,7){$\underbrace{\phantom{AAAA}}_{\text {n}}$}
\put(83,38){{\tiny $1$}}
\put(90,34){\circle{6}}
  \put(120,40){{\tiny $1$}}
\put(125,34){\circle{6}}
\put(100,31){$\cdots$}
\put(94,34){\line(1,0){4}}
\put(117,34){\line(1,0){4}}
\put(83,40){$\overbrace{\phantom{AAAAA}}^{\text {$\ell$}}$}
\end{picture}
\vspace{5mm}

%%%%%%%%%%%%%%%%%%%%%%%% Bmn %%%%%%%%%%%%%%%%%%%%%%%%%%%%%%%%%%
\item $B_{m,n}(m\ge0, n\ge 3)$

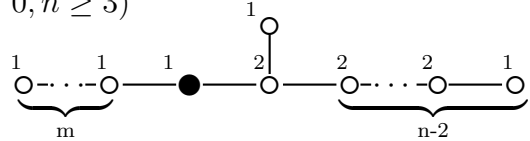
\begin{picture}(400,35)(-20,0)
    \thicklines

\put(23,18){{\tiny $1$}}
\put(28,12){\circle{6}}
\put(33,12){\line(1,0){4}}  
\put(37,9){$\cdots$}
\put(26,7){$\underbrace{\phantom{AAAA}}_{\text {m}}$} 
  
\put(51,12){\line(1,0){4}}
\put(55,18){{\tiny $1$}}
\put(60,12){\circle{6}}
\put(65,12){\line(1,0){20}}
\put(80,18){{\tiny $1$}}
\put(90,12){\circle*{8}}
\put(95,12){\line(1,0){20}}
\put(114,18){{\tiny $2$}}
\put(120,12){\circle{6}}
\put(125,12){\line(1,0){20}}
\put(145,18){{\tiny $2$}}
\put(150,12){\circle{6}}
\put(154,12){\line(1,0){4}}

\put(158.5,9){$\cdots$}
\put(146,7){$\underbrace{\phantom{AAAAAAAA}}_{\text {n-2}}$}

\put(175,12){\line(1,0){4}}
\put(177,18){{\tiny $2$}}
\put(183,12){\circle{6}}
\put(187,12){\line(1,0){20}}
\put(207,18){{\tiny $1$}}
\put(212,12){\circle{6}}

\put(120,16){\line(0,1){15}}
\put(111,38){{\tiny $1$}}
\put(120,34){\circle{6}}
\end{picture}

\vspace{5mm}

%%%%%%%%%%%%%%%%%%%%%%%% Cmn %%%%%%%%%%%%%%%%%%%%%%%%%%%%%%%%%%
\item $C_{m,n}(m\ge 0, n\ge 4)$

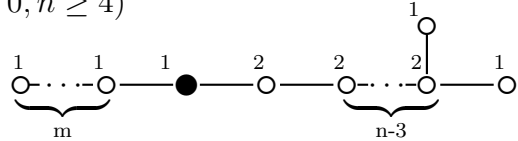
\begin{picture}(400,35)(-20,0)
    \thicklines
    
  \put(25,18){{\tiny $1$}}
\put(28,12){\circle{6}}
\put(31,12){\line(1,0){4}}

\put(26,7){$\underbrace{\phantom{AAAA}}_{\text {m}}$}
\put(37,9){$\cdots$}
 
 \put(51,12){\line(1,0){4}}
  \put(55,18){{\tiny $1$}}
\put(60,12){\circle{6}}
\put(65,12){\line(1,0){20}}
  \put(80,18){{\tiny $1$}}
\put(90,12){\circle*{8}}
\put(95,12){\line(1,0){20}}
  \put(115,18){{\tiny $2$}}
\put(120,12){\circle{6}}
\put(125,12){\line(1,0){20}}
  \put(145,18){{\tiny $2$}}
\put(150,12){\circle{6}}
\put(154,12){\line(1,0){4}}

\put(158,9){$\cdots$}
\put(149,7){$\underbrace{\phantom{AAAA}}_{\text {n-3}}$}
 
  \put(172,12){\line(1,0){4}}
  \put(174,18){{\tiny $2$}}
\put(180,12){\circle{6}}
\put(185,12){\line(1,0){20}}
  \put(205,18){{\tiny $1$}}
\put(210,12){\circle{6}}
\put(180,16){\line(0,1){15}}
  \put(173,38){{\tiny $1$}}
\put(180,34){\circle{6}}
\end{picture}

\vspace{5mm}
   
%%%%%%%%%%%%%%%%%%%%%%%% Dn %%%%%%%%%%%%%%%%%%%%%%%%%%%%%%%%%%
\item $D_{n}(n\ge 0)$

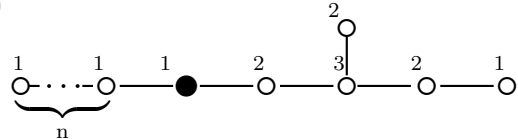
\begin{picture}(400,35)(-20,0)
    \thicklines
    
  \put(25,18){{\tiny $1$}}
\put(28,12){\circle{6}}
\put(31,12){\line(1,0){4}}

\put(26,7){$\underbrace{\phantom{AAAA}}_{\text {n}}$}
\put(37,9){$\cdots$}

\put(37,9){$\cdots$}
 
 \put(51,12){\line(1,0){4}}
  \put(55,18){{\tiny $1$}}
\put(60,12){\circle{6}}
\put(65,12){\line(1,0){20}}
  \put(80,18){{\tiny $1$}}
\put(90,12){\circle*{8}}
\put(95,12){\line(1,0){20}}
  \put(115,18){{\tiny $2$}}
\put(120,12){\circle{6}}
\put(125,12){\line(1,0){20}}
  \put(145,18){{\tiny $3$}}
\put(150,12){\circle{6}}
\put(154,12){\line(1,0){20}}
  \put(174,18){{\tiny $2$}}
\put(180,12){\circle{6}}
\put(185,12){\line(1,0){20}}
  \put(205,18){{\tiny $1$}}
\put(210,12){\circle{6}}

\put(150,16){\line(0,1){15}}
  \put(143,38){{\tiny $2$}}
\put(150,34){\circle{6}}
\end{picture}

\vspace{5mm}

%%%%%%%%%%%%%%%%%%%%%%%% Fn %%%%%%%%%%%%%%%%%%%%%%%%%%%%%%%%%%

\item $F_{n}(n\ge 0)$

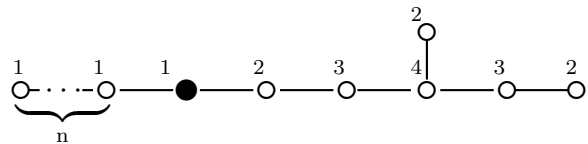
\begin{picture}(400,35)(-20,0)
    \thicklines
    
  \put(25,18){{\tiny $1$}}
\put(28,12){\circle{6}}
\put(31,12){\line(1,0){4}}

\put(26,7){$\underbrace{\phantom{AAAA}}_{\text {n}}$}
\put(37,9){$\cdots$}

 \put(51,12){\line(1,0){4}}
  \put(55,18){{\tiny $1$}}
\put(60,12){\circle{6}}
\put(65,12){\line(1,0){20}}
  \put(80,18){{\tiny $1$}}
\put(90,12){\circle*{8}}
\put(95,12){\line(1,0){20}}
  \put(115,18){{\tiny $2$}}
\put(120,12){\circle{6}}
\put(125,12){\line(1,0){20}}
  \put(145,18){{\tiny $3$}}
\put(150,12){\circle{6}}
\put(154,12){\line(1,0){20}}
  \put(174,18){{\tiny $4$}}
\put(180,12){\circle{6}}
\put(185,12){\line(1,0){20}}
  \put(205,18){{\tiny $3$}}
\put(210,12){\circle{6}}
\put(213,12){\line(1,0){20}}
  \put(232,18){{\tiny $2$}}
\put(236,12){\circle{6}}
\put(180,16){\line(0,1){15}}
  \put(174,38){{\tiny $2$}}
\put(180,34){\circle{6}}
\end{picture}

\vspace{5mm}

%%%%%%%%%%%%%%%%%%%%%%%% Hn %%%%%%%%%%%%%%%%%%%%%%%%%%%%%%%%%%

\item $H_{n}(n\ge 5)$

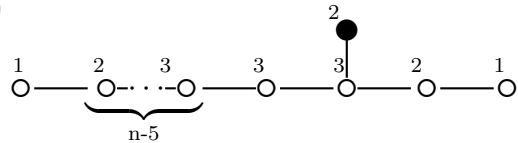
\begin{picture}(400,35)(-20,0)
    \thicklines
    
  \put(25,18){{\tiny $1$}}
\put(28,12){\circle{6}}
\put(33,12){\line(1,0){20}}
  \put(55,18){{\tiny $2$}}
\put(60,12){\circle{6}}
%\put(65,12){\line(1,0){20}}
\put(68,9){$\cdots$}
  \put(80,18){{\tiny $3$}}
\put(90,12){\circle{6}}

\put(64,12){\line(1,0){4}}

%\put(98,9){$\cdots$}
\put(96,12){\line(1,0){20}}
 \put(52,7){$\underbrace{\phantom{AAAAA}}_{\text {n-5}}$}

% \put(94,9){$\underbrace{\cdots}_{\text {n-5}}$}
 
\put(82,12){\line(1,0){4}}

  \put(115,18){{\tiny $3$}}
\put(120,12){\circle{6}}
\put(125,12){\line(1,0){20}}
  \put(145,18){{\tiny $3$}}
\put(150,12){\circle{6}}
\put(154,12){\line(1,0){20}}
  \put(174,18){{\tiny $2$}}
\put(180,12){\circle{6}}
\put(185,12){\line(1,0){20}}
  \put(205,18){{\tiny $1$}}
\put(210,12){\circle{6}}

\put(150,16){\line(0,1){15}}
  \put(143,38){{\tiny $2$}}
\put(150,34){\circle*{8}}
\end{picture}

\vspace{3mm}

%%%%%%%%%%%%%%%%%%%%%%%% Gamma1 %%%%%%%%%%%%%%%%%%%%%%%%%%%%%%%%%%
\item $\Gamma_{1}$

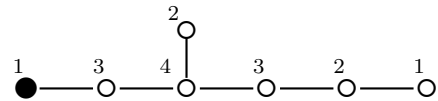
\begin{picture}(400,35)(-20,0)
    \thicklines
    
  \put(25,18){{\tiny $1$}}
\put(30,12){\circle*{8}}
\put(35,12){\line(1,0){20}}
  \put(55,18){{\tiny $3$}}
\put(60,12){\circle{6}}
\put(65,12){\line(1,0){20}}
  \put(80,18){{\tiny $4$}}
\put(90,12){\circle{6}}
\put(95,12){\line(1,0){20}}
  \put(115,18){{\tiny $3$}}
\put(120,12){\circle{6}}
\put(125,12){\line(1,0){20}}
  \put(145,18){{\tiny $2$}}
\put(150,12){\circle{6}}
\put(155,12){\line(1,0){20}}
  \put(175,18){{\tiny $1$}}
\put(180,12){\circle{6}}

\put(90,16){\line(0,1){15}}
  \put(83,38){{\tiny $2$}}
\put(90,34){\circle{6}}
\end{picture}

\vspace{5mm}

%%%%%%%%%%%%%%%%%%%%%%%% Gamma2 %%%%%%%%%%%%%%%%%%%%%%%%%%%%%%%%%%
\item $\Gamma_{2}$

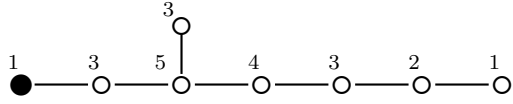
\begin{picture}(400,35)(-20,0)
    \thicklines

  \put(25,18){{\tiny $1$}}
\put(30,12){\circle*{8}}
\put(35,12){\line(1,0){20}}
  \put(55,18){{\tiny $3$}}
\put(60,12){\circle{6}}
\put(65,12){\line(1,0){20}}
  \put(80,18){{\tiny $5$}}
\put(90,12){\circle{6}}
\put(95,12){\line(1,0){20}}
  \put(115,18){{\tiny $4$}}
\put(120,12){\circle{6}}
\put(125,12){\line(1,0){20}}
  \put(145,18){{\tiny $3$}}
\put(150,12){\circle{6}}
\put(155,12){\line(1,0){20}}
  \put(175,18){{\tiny $2$}}
\put(180,12){\circle{6}}
\put(185,12){\line(1,0){20}}
  \put(205,18){{\tiny $1$}}
\put(210,12){\circle{6}}
\put(90,16){\line(0,1){15}}
  \put(83,38){{\tiny $3$}}
\put(90,34){\circle{6}}
\end{picture}

\vspace{5mm}

%%%%%%%%%%%%%%%%%%%%%%%% Gamma3 %%%%%%%%%%%%%%%%%%%%%%%%%%%%%%%%%%
\item $\Gamma_{3}$

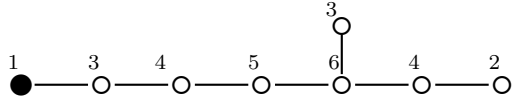
\begin{picture}(400,35)(-20,0)
    \thicklines

  \put(25,18){{\tiny $1$}}
\put(30,12){\circle*{8}}
\put(35,12){\line(1,0){20}}
  \put(55,18){{\tiny $3$}}
\put(60,12){\circle{6}}
\put(65,12){\line(1,0){20}}
  \put(80,18){{\tiny $4$}}
\put(90,12){\circle{6}}
\put(95,12){\line(1,0){20}}
  \put(115,18){{\tiny $5$}}
\put(120,12){\circle{6}}
\put(125,12){\line(1,0){20}}
  \put(145,18){{\tiny $6$}}
\put(150,12){\circle{6}}
\put(155,12){\line(1,0){20}}
  \put(175,18){{\tiny $4$}}
\put(180,12){\circle{6}}
\put(185,12){\line(1,0){20}}
  \put(205,18){{\tiny $2$}}
\put(210,12){\circle{6}}
\put(150,16){\line(0,1){15}}
  \put(144,38){{\tiny $3$}}
\put(150,34){\circle{6}}
\end{picture}
\end{enumerate}

\vskip 3mm
In what follows, using the classification of rational triple points as above,  
we determine the set of Ulrich ideals ${\mathcal X}_A$. 
Indeed, Tyurina \cite{Tyu} gave the defining ideal $I_2(\mathbb{M})$ 
for any rational triple point $A$.  
Since any rational triple point has $\CMtype(A)=2$, we can apply 
Lemma 2.7 and determine $\tr_A(\omega_A)$ explicitly. 
Then we can prove that $\tr_A(\omega_A)$ is an Ulrich ideal.  
Indeed, it is the minimum Ulrich ideal.  
Thus in order to determine ${\mathcal X}_A$, it suffices to show 
that any ideal containing $\tr_A(\omega_A)$ is an Ulrich ideal. 
 
\par 
First we consider the case $H=A_{\ell ,m,n}$. 

%  A_{m,n}
\begin{prop} \label{Aprop}
Suppose $H=A_{\ell ,m,n}$ $(0\le \ell \le m \le n)$ and 
let $A$ be the $(x,y,z,t)$-adic completion of $R(H)$, where 
\begin{equation*}
\begin{aligned}
R(H)&= k[t,x,y,z]/(xy-t^{\ell+m+2},xz-t^{n+2}-zt^{\ell+1},yz+yt^{n+1}-zt^{m+1})\\ 
&= k[t,x,y,z]/I_2
\begin{pmatrix}
x & t^{m+1}& t^{n+1}+z \\[2mm]
t^{\ell+1} & y & z \\
\end{pmatrix}_.
\end{aligned}
\end{equation*}
Then 
\begin{enumerate}
\item$\tr_{A}(\omega_{A})=(x,y,z,t^{\ell +1})$. 
\item${\mathcal X}_{A}=\{(x,y,z,t^i)|i=1,2,\ldots ,\ell +1\}$. 
\item $A$ is nearly Gorenstein if and only if $\ell=0$, that is, 
$A$ is a cyclic quotient singularity. 
\end{enumerate}
\end{prop}

\begin{proof} 
(1) By Lemma \ref{Burch}, we have  
\[
\tr_A(\omega_{A})=(x,y,z,t^{\ell+1},t^{m+1},t^{n+1}+z)=(x,y,z,t^{\ell+1}). 
\]
Moreover, (3) follows from here. 
\par 
(2) (1) and Proposition \ref{Ulrich-Can} imply 
${\mathcal X}_A \subset \{(x,y,z,t^{i}) \,|\, (i=1,2,\ldots,\ell +1) \}$. 
Let us show the converse. 
Fix $i$ with $1 \le i \le \ell+1$ and  
put $J_i=(x,y,z,t^{i})$, $Q_i=(t^i,x+y+z)$. 
Then it is enough to show that $J_i$ is an Ulrich ideal. 
\begin{eqnarray*}
x^2\!\! &= &\!\! (x+y+z)x-xy-xz=(x+y+z)x-t^{l+1}(t^{m+1}+t^{n+1}+z) \in Q_iJ_i. \\
xy\!\! & = &\!\!  t^{\ell+1}t^{m+1} \in Q_iJ_i. \\
y^2\!\! & = &\!\! y(x+y+z)-xy-yz=y(x+y+z)-t^{\ell+1}(t^{m+1}+t^{m-\ell}z-t^{n-\ell}y) \in Q_iJ_i.
\end{eqnarray*}
Hence $J_i^2=Q_iJ_i$. 
Furthermore, we have
\[
e_0(J_i)= \ell(A/Q_i)\ell _A(k[[t,y,z]]/(t^i, y^2, yz, z^2))=3i
=(\mu(J_i)-1)\ell(A/J_i). 
\]
%\begin{eqnarray*}
%e_0(J_i)&=& \ell(A/Q_i)
%=\ell _A(k[t,y,z]/(t^i, y^2, yz, z^2))=3i,  \\
%(\mu(J_i)-1)\ell(A/J_i)&=& (4-1)\ell(k[t]/(t^i))=3i. 
%\end{eqnarray*}
Thus $J_i$ is an Ulrich ideal by virtue of Lemma \ref{CriterionUlrich}. 
\end{proof}

%  B_{m,n}
\begin{prop} \label{Bprop}
Suppose $H=B_{m,n}$ $(m \ge 0, \, n \ge 3)$ and 
let $A$ be the $(x,y,z,t)$-adic completion of $R(H)$, where 
$R(H)$ is given as follows. 
\par \vspace{1mm} 
when $n=2k-1,k\ge 2$, 
\begin{equation*}
\begin{aligned}
R(H)
&= k[t,x,y,z]/(xz-yt^{m+1},xy-t^{m+k+1}-zt^{m+2},y^2-zt^{k}-z^2t) \\[2mm]
&= k[t,x,y,z]/I_2
{\small \begin{pmatrix}
x & y & t^k+zt\\[1mm]
t^{m+1} & z & y
\end{pmatrix}}_.
\end{aligned}
\end{equation*}
\par \vspace{1mm}
when $n=2k,k\ge 2$, 
\begin{equation*}
\begin{aligned}
R(H)&= k[t,x,y,z]/(xz-yt^{m+1},xy+xt^{k}-zt^{m+2},y^2+yt^{k}-z^2t)\\[2mm] 
&= k[t,x,y,z]/I_2
{\small \begin{pmatrix}
x & y & zt \\[1mm]
t^{m+1} & z &y+t^k\\
\end{pmatrix}}_.
\end{aligned}
\end{equation*}
Then 
\begin{enumerate}
\item$\tr_{A}(\omega_{A})=(x,y,z,t^a)$, where $a=min\{k,m+1\}$. 
\item${\mathcal X}_{A}=\{(x,y,z,t^i)\,|\, i=1,2,\ldots ,a\}$. 
\item $A$ is nearly Gorenstein if and only if $m=0$. 
\end{enumerate}
\end{prop}

\begin{proof}
We may assume $n=2k-1$, $k \ge 2$.  
Similarly one can prove the assertion in the case of $n=2k$.  
\par \vspace{1mm}
(1) By Lemma \ref{Burch}, we have  
\[
\tr_A(\omega_{A})=(x,y,t^k+zt, t^{m+1}, z,y)=(x,y,z,t^a), 
\]
where $a=\min\{k,m+1\}$. 
Moreover, (3) follows from here. 
\par \vspace{1mm}
(2) (1) and Proposition \ref{Ulrich-Can} imply 
${\mathcal X}_A \subset \{(x,y,z,t^{i}) \,|\, (i=1,2,\ldots,a) \}$. 
Let us show the converse. 
Fix $i$ with $1 \le i \le a$ and  
put $J_i=(x,y,z,t^{i})$, $Q_i=(t^i,x+z)$. 
Then it is enough to show that $J_i$ is an Ulrich ideal. 
In fact, one can easily see that $y^2,yz,z^2 \in Q_iJ_i$ and thus 
$J_i^2=Q_iJ_i$ as follows:
\begin{eqnarray*}
y^2 &= & zt^k + z^2t=zt^k+z(x+z)t-zxt=zt^k+z(x+z)t-yt^{m+2} \in Q_iJ_i. \\
yz & = &  y(x+z)-xy=y(x+z)-t^{m+k+1}+zt^{m+2} \in Q_iJ_i. \\
z^2 & = & z(x+z)-xz=z(x+z)-yt^{m+1} \in Q_iJ_i. 
\end{eqnarray*}
Moreover, we have 
$e_0(J_i)=3i=(\mu(J_i)-1)\ell(A/J_i)$. 
\end{proof}

\par \vspace{3mm}
%  C_{m,n}
\begin{prop} \label{Cprop}
Suppose $H=C_{m,n}$ $(m \ge 0,\,n \ge 4)$ and 
let $A$ be the $(x,y,z,t)$-adic completion of $R(H)$, where 
\begin{equation*}
\begin{aligned}
R(H)&=k[t,x,y,z]/(xz-yt^{m+1},xy-t^{m+3}-z^{n-1}t^{m+1},y^2-zt^{2}-z^{n})\\[2mm] 
&= k[t,x,y,z]/I_2
{\small \begin{pmatrix}
x & y & t^2+z^{n-1} \\[1mm]
t^{m+1} & z & y
\end{pmatrix}}_.
\end{aligned}
\end{equation*}
Then 
\begin{enumerate}
\item$\tr_{A}(\omega_{A})=(x,y,z,t^a)$, where $a=min\{2,m+1\}$. 
\item${\mathcal X}_{A}=\{(x,y,z,t^i)\,|\, i=1,2,\ldots ,a\}$. 
\item $A$ is nearly Gorenstein if and only if $m=0$. 
\end{enumerate}
\end{prop}

\begin{proof}
(1) By Lemma \ref{Burch}, we have  
\[
\tr_A(\omega_{A})=(x,y,t^2+z^{n-1}, t^{m+1}, z, y)=(x,y,z,t^a), 
\]
where $a=\min\{2,m+1\}$. 
Moreover, (3) follows from here. 
\par \vspace{1mm}
(2) (1) and Proposition \ref{Ulrich-Can} imply 
${\mathcal X}_A \subset \{(x,y,z,t^{i}) \,|\, (i=1,2,\ldots,a) \}$. 
Let us show the converse. 
Fix $i$ with $1 \le i \le a$ and  
put $J_i=(x,y,z,t^{i})$, $Q_i=(t^i,x+z)$. 
Then it is enough to show that $J_i$ is an Ulrich ideal. 
In fact, one can easily see that $y^2,yz,z^2 \in Q_iJ_i$  and thus 
$J_i^2=Q_iJ_i$ as follows:
\begin{eqnarray*}
z^2 & = & z(x+z)-xz=z(x+z)-yt^{m+1} \in Q_iJ_i.\\
y^2 &= & zt^2+ z^n \in Q_iJ_i. \\
yz & = &  y(x+z)-xy=y(x+z)-t^{m+3}-z^{n-1}t^{m+1} \in Q_iJ_i. \\
\end{eqnarray*}
Moreover, we have 
$e_0(J_i)=3i=(\mu(J_i)-1)\ell(A/J_i)$. 
\end{proof}

\par \vspace{3mm}
%  D_{n}
\begin{prop} \label{Dprop}
Suppose $H=D_{n}$ $(n \ge 0)$ and 
let $A$ be the $(x,y,z,t)$-adic completion of $R(H)$, where 
\begin{equation*}
\begin{aligned}
R(H)&=k[t,x,y,z]/(xz-yt^{n+1},xy+xt^2-z^{2}t^{n+1},y^2+yt^2-z^3) \\[2mm]
&= k[t,x,y,z]/I_2
{\small \begin{pmatrix}
x & y & z^2 \\[1mm]
t^{n+1} & z & y+t^2 
\end{pmatrix}}_.
\end{aligned}
\end{equation*}
Then 
\begin{enumerate}
\item$\tr_{A}(\omega_{A})=(x,y,z,t^a)$, where $a=min\{2,n+1\}$. 
\item${\mathcal X}_{A}=\{(x,y,z,t^i)\,|\, i=1,2,\ldots ,a\}$. 
\item $A$ is nearly Gorenstein if and only if $n=0$. 
\end{enumerate}
\end{prop}

\begin{proof}
(1) By Lemma \ref{Burch}, we have  
\[
\tr_A(\omega_{A})=(x,y,z^2, t^{n+1}, z, y+t^2)=(x,y,z,t^a), 
\]
where $a=\min\{2,n+1\}$. 
Moreover, (3) follows from here. 
\par \vspace{1mm}
(2) (1) and Proposition \ref{Ulrich-Can} imply 
${\mathcal X}_A \subset \{(x,y,z,t^{i}) \,|\, (i=1,2,\ldots,a) \}$. 
Let us show the converse. 
Fix $i$ with $1 \le i \le a$ and  
put $J_i=(x,y,z,t^{i})$, $Q_i=(t^i,x+z)$. 
Then it is enough to show that $J_i$ is an Ulrich ideal. 
In fact, one can easily see that $y^2,yz,z^2 \in Q_iJ_i$  and thus 
$J_i^2=Q_iJ_i$ as follows:
\begin{eqnarray*}
z^2 & = & z(x+z)-xz=z(x+z)-yt^{n+1} \in Q_iJ_i.\\
y^2 &= & z^3-yt^2 \in Q_iJ_i. \\
yz & = &  y(x+z)-xy=y(x+z)-z^2t^{n+1}+xt^2 \in Q_iJ_i. \\
\end{eqnarray*}
Moreover, we have 
$e_0(J_i)=3i=(\mu(J_i)-1)\ell(A/J_i)$. 
\end{proof}

\par \vspace{3mm}
%  F_{n}
\begin{prop} \label{Fprop}
Suppose $H=F_{n}$ $(n \ge 0)$ and 
let $A$ be the $(x,y,z,t)$-adic completion of $R(H)$, where 
\begin{equation*}
\begin{aligned}
R(H)&=k[t,x,y,z]/(xz-yt^{n+1},xy-t^{n+4}-z^{2}t^{n+1},y^2-zt^{3}-z^3)\\[2mm]
&= k[t,x,y,z]/I_2
{\small \begin{pmatrix}
x & y & t^3+z^2 \\[1mm]
t^{n+1} & z & y
\end{pmatrix}}_.
\end{aligned}
\end{equation*}
Then 
\begin{enumerate}
\item$\tr_{A}(\omega_{A})=(x,y,z,t^a)$, where $a=min\{3,n+1\}$. 
\item${\mathcal X}_{A}=\{(x,y,z,t^i)\,|\, i=1,2,\ldots ,a\}$. 
\item $A$ is nearly Gorenstein if and only if $n=0$. 
\end{enumerate}
\end{prop}

\begin{proof}
(1) By Lemma \ref{Burch}, we have  
\[
\tr_A(\omega_{A})=(x,y,z^2, t^3+z^2, t^{n+1}, z, y)=(x,y,z,t^a), 
\]
where $a=\min\{3,n+1\}$. 
Moreover, (3) follows from here. 
\par \vspace{1mm}
(2) (1) and Proposition \ref{Ulrich-Can} imply 
${\mathcal X}_A \subset \{(x,y,z,t^{i}) \,|\, (i=1,2,\ldots,a) \}$. 
Let us show the converse. 
Fix $i$ with $1 \le i \le a$ and  
put $J_i=(x,y,z,t^{i})$, $Q_i=(t^i,x+z)$. 
Then it is enough to show that $J_i$ is an Ulrich ideal. 
In fact, one can easily see that $y^2,yz,z^2 \in Q_iJ_i$ and thus 
$J_i^2=Q_iJ_i$  as follows:
\begin{eqnarray*}
z^2 & = & z(x+z)-xz=z(x+z)-yt^{n+1} \in Q_iJ_i.\\
y^2 &= & z^3+zt^3 \in Q_iJ_i. \\
yz & = &  y(x+z)-xy=y(x+z)-t^{n+4}-z^2t^{n+1} \in Q_iJ_i. \\
\end{eqnarray*}
Moreover, we have 
$e_0(J_i)=3i=(\mu(J_i)-1)\ell(A/J_i)$. 
\end{proof}

\par \vspace{3mm}
%  H_{n}
\begin{prop} \label{Hprop}
Suppose $H=H_{n}$ $(n \ge 5)$ and 
let $A$ be the $(x,y,z,t)$-adic completion of $R(H)$, where 
$R(H)$ is given as follows: 
\par 
when $n=3k-1,k\ge2$, \\
\begin{equation*}
\begin{aligned}
R(H_{n})&=k[t,x,y,z]/(x^2-yzt-yt^{k},xy-z^{2}t-zt^{k},y^2-xz)\\[2mm]
&= k[t,x,y,z]/I_2
{\small \begin{pmatrix}
x & y & zt+t^k\\[1mm]
y & z & x\\
\end{pmatrix}}_.
\end{aligned}
\end{equation*}
\par 
when $n=3k ,k\ge2$, \\
\begin{equation*}
\begin{aligned}
R(H_{n})&=k[t,x,y,z]/(x^2+xt^{k}-yzt,xy-z^{2}t+yt^{k},y^2-xz)\\[2mm] 
&= k[t,x,y,z]/I_2
{\small \begin{pmatrix}
x & y & zt\\[1mm]
y & z & x+t^k\\
\end{pmatrix}}_.
\end{aligned}
\end{equation*}
\par
when $n=3k+1 ,k\ge2$, \\
\begin{equation*}
\begin{aligned}
R(H_{n})&=k[t,x,y,z]/(x^2-yzt-zt^{k+1},xy-z^{2}t,y^2-xz+yt^{k})\\[2mm] 
&= k[t,x,y,z]/I_2
{\small \begin{pmatrix}
x & y & zt \\[1mm]
y+t^k & z & x\\
\end{pmatrix}}_.
\end{aligned}
\end{equation*}
Then 
\begin{enumerate}
\item$\tr_{A}(\omega_{A})=(x,y,z,t^k)$. 
\item${\mathcal X}_{A}=\{(x,y,z,t^i)\,|\, i=1,2,\ldots ,k\}$. 
\item $A$ is not nearly Gorenstein. 
\end{enumerate}
\end{prop}

\begin{proof}
We may assume $n=3k-1$, $k \ge 2$.  
Similarly one can prove the assertion in the case of $n=3k$ or $3k+1$.  
(1) By Lemma \ref{Burch}, we have  
\[
\tr_A(\omega_{A})=(x,y,zt+t^k, y, z, x)=(x,y,z,t^k). 
\] 
Moreover, (3) follows from here. 
\par \vspace{1mm}
(2) (1) and Proposition \ref{Ulrich-Can} imply 
${\mathcal X}_A \subset \{(x,y,z,t^{i}) \,|\, (i=1,2,\ldots,k) \}$. 
Let us show the converse. 
Fix $i$ with $1 \le i \le k$ and  
put $J_i=(x,y,z,t^{i})$, $Q_i=(t^i, z)$. 
Then it is enough to show that $J_i$ is an Ulrich ideal. 
In fact, one can easily see that $y^2,yz,z^2 \in Q_iJ_i$  and thus 
$J_i^2=Q_iJ_i$ as follows:
\begin{eqnarray*}
x^2 & = & yzt+t^k \in Q_iJ_i\\
y^2 &= & xz \in Q_iJ_i \\
xy & = &  y(x+z)-xy=z^2t+zt^k \in Q_iJ_i \\
\end{eqnarray*}
Moreover, we have 
$e_0(J_i)=3i=(\mu(J_i)-1)\ell(A/J_i)$. 
\end{proof}

\par \vspace{3mm}
%  \Gamma_{1}
\begin{prop} \label{G1prop}
Suppose $H=\Gamma_1$ and 
let $A$ be the $(x,y,z,t)$-adic completion of $R(H)$, where 
\begin{equation*}
\begin{aligned}
R(H)&=k[t,x,y,z]/(x^2-yt^2+xz^2,xy-zt^2+yz^2,y^2-xz)\\[2mm] 
&= k[t,x,y,z]/I_2
{\small \begin{pmatrix}
x & y & t^2\\[1mm]
y & z & x+z^2\\
\end{pmatrix}}_.
\end{aligned}
\end{equation*}
Then 
\begin{enumerate}
\item$\tr_{A}(\omega_{A})=(x,y,z,t^2)$. 
\item${\mathcal X}_{A}=\{(x,y,z,t^i)\,|\, i=1,2\}$. 
\item $A$ is not nearly Gorenstein. 
\end{enumerate}
\end{prop}

\begin{proof}
(1) By Lemma \ref{Burch}, we have  
\[
\tr_A(\omega_{A})=(x,y,t^2, y, z, x+z^2)=(x,y,z,t^2). 
\] 
Moreover, (3) follows from here. 
\par \vspace{1mm}
(2) (1) and Proposition \ref{Ulrich-Can} imply 
${\mathcal X}_A \subset \{(x,y,z,t^{i}) \,|\, (i=1,2) \}$. 
Let us show the converse. 
Fix $i$ with $I=1,2$ and  
put $J_i=(x,y,z,t^{i})$, $Q_i=(t^i, z)$. 
Then it is enough to show that $J_i$ is an Ulrich ideal. 
In fact, one can easily see that $y^2,yz,z^2 \in Q_iJ_i$  and thus 
$J_i^2=Q_iJ_i$ as follows:
\begin{eqnarray*}
x^2 & = & yt^2-xz^2 \in Q_iJ_i.\\
xy &= & zt^2-yz^2 \in Q_iJ_i. \\
y^2 & = &  xz \in Q_iJ_i. \\
\end{eqnarray*}
Moreover, we have 
$e_0(J_i)=3i=(\mu(J_i)-1)\ell(A/J_i)$. 
\end{proof}

\par \vspace{3mm}
%  \Gamma_{2}
\begin{prop} \label{G2prop}
Suppose $H=\Gamma_2$ and 
let $A$ be the $(x,y,z,t)$-adic completion of $R(H)$, where 
\begin{equation*}
\begin{aligned}
R(\Gamma_{2})&=k[t,x,y,z]/(x^2-yz^2+xt^2,xy-z^3+yt^2,y^2-xz)\\[2mm] 
&= k[t,x,y,z]/I_2
{\small \begin{pmatrix}
x & y & z^2 \\[1mm]
y & z &x+t^2
\end{pmatrix}}_.
\end{aligned}
\end{equation*}
Then 
\begin{enumerate}
\item$\tr_{A}(\omega_{A})=(x,y,z,t^2)$. 
\item${\mathcal X}_{A}=\{(x,y,z,t^i)\,|\, i=1,2 \}$. 
\item $A$ is not nearly Gorenstein. 
\end{enumerate}
\end{prop}

\begin{proof}
(1) By Lemma \ref{Burch}, we have  
\[
\tr_A(\omega_{A})=(x,y,t^2, y, z, x+z^2)=(x,y,z,t^2). 
\] 
Moreover, (3) follows from here. 
\par \vspace{1mm}
(2) (1) and Proposition \ref{Ulrich-Can} imply 
${\mathcal X}_A \subset \{(x,y,z,t^{i}) \,|\, (i=1,2) \}$. 
Let us show the converse. 
Fix $i$ with $I=1,2$ and  
put $J_i=(x,y,z,t^{i})$, $Q_i=(t^i, z)$. 
Then it is enough to show that $J_i$ is an Ulrich ideal. 
In fact, one can easily see that $y^2,yz,z^2 \in Q_iJ_i$  and thus 
$J_i^2=Q_iJ_i$ as follows:
\begin{eqnarray*}
x^2 & = & yz^2-xt^2 \in Q_iJ_i.\\
xy &= & z^3-yt^2 \in Q_iJ_i. \\
y^2 & = &  xz \in Q_iJ_i. \\
\end{eqnarray*}
Moreover, we have 
$e_0(J_i)=3i=(\mu(J_i)-1)\ell(A/J_i)$. 
\end{proof}

\par \vspace{3mm}
%  \Gamma_{3}
\begin{prop}  \label{G3prop}
Suppose $H=\Gamma_2$ and 
let $A$ be the $(x,y,z,t)$-adic completion of $R(H)$, where 
\begin{equation*}
\begin{aligned}
R(\Gamma_{3})&=k[t,x,y,z]/(x^2-yt^2-yz^3,xy-zt^2-z^4,y^2-xz)\\[2mm] 
&= k[t,x,y,z]/I_2
{\small \begin{pmatrix}
x & y & t^2+z^3 \\[1mm]
y & z & x\\
\end{pmatrix}}_.
\end{aligned}
\end{equation*}
Then 
\begin{enumerate}
\item$\tr_{A}(\omega_{A})=(x,y,z,t^2)$. 
\item${\mathcal X}_{A}=\{(x,y,z,t^i)\,|\, i=1,2 \}$. 
\item $A$ is not nearly Gorenstein. 
\end{enumerate}
\end{prop}

\begin{proof}
(1) By Lemma \ref{Burch}, we have  
\[
\tr_A(\omega_{A})=(x,y,t^2+z^3, y, z, x)=(x,y,z,t^2). 
\] 
Moreover, (3) follows from here. 
\par \vspace{1mm}
(2) (1) and Proposition \ref{Ulrich-Can} imply 
${\mathcal X}_A \subset \{(x,y,z,t^{i}) \,|\, (i=1,2) \}$. 
Let us show the converse. 
Fix $i$ with $I=1,2$ and  
put $J_i=(x,y,z,t^{i})$, $Q_i=(t^i, z)$. 
Then it is enough to show that $J_i$ is an Ulrich ideal. 
In fact, one can easily see that $y^2,yz,z^2 \in Q_iJ_i$  and thus 
$J_i^2=Q_iJ_i$ as follows:
\begin{eqnarray*}
x^2 & = & yt^2+yz \in Q_iJ_i.\\
xy &= & zt^2+z^4 \in Q_iJ_i. \\
y^2 & = &  xz \in Q_iJ_i. \\
\end{eqnarray*}
Moreover, we have 
$e_0(J_i)=3i=(\mu(J_i)-1)\ell(A/J_i)$. 
\end{proof}

\par 
Summarizing the above argument, we can prove the following corollary 
and Theorem \ref{Main-RTP}.

%%%  Corollary 
\begin{cor} \label{RTP-NG}
Let $A$ be a rational triple point, 
and let ${\mathcal X}_A$ be  
the set of Ulrich ideals. 
Then 
\begin{enumerate}
\item $\res(A)=\sharp({\mathcal X}_A)$ and 
\[
\res(A)= 
\left\{
\begin{array}{lll}
\ell+1 & \text{if}\; H=A_{\ell,m,n} & (0 \le \ell \le m \le n) \\
\min\{k,m+1\} & \text{if}\; H=B_{m,n} & (n=2k-1, 2k, \; k \ge 2)\\
\min\{2,m+1\} & \text{if}\; H=C_{m,n} & (m \ge 0, n \ge 4) \\
\min\{2,n+1\} & \text{if}\; H=D_{n} & (n \ge 0) \\
\min\{3,n+1\} & \text{if}\; H=F_{n} & (n \ge 0) \\
k & \text{if}\; H=H_{n} & (n =3k-1,3k,3k+1,\, k \ge 2) \\
2 & \text{if}\; H=\Gamma_1,\Gamma_2,\Gamma_3
\end{array}
\right.
\]
\item The following conditions are equivalent. 
\begin{enumerate}
\item ${\mathcal X}_A=\{\m\}$. 
\item $A$ is nearly Gorenstein. 
\item $H$ is one of $A_{0,m,n}$ $(1 \le m \le n)$, 
\, $B_{0,n}$ $(n \ge 3)$,
\, $C_{m,0}$ $(m \ge 4)$, 
\, $D_0$, 
\, $F_0$. 
\end{enumerate}
\end{enumerate}
\end{cor}

%%%%%%%%%%%%%%%%%%%%%%%%%%%%%%%%%%%%%%%%%%%%%%%%%%%%%%%%%%%
%%%% Section 4
\section{Quotient singularities}\label{s:nHP}

In this section, we classify all Ulrich ideals in a two-dimensional quotient singularity. 

\par 
Let $A$ be a two-dimensional quotient singularity of $e_0(A)=e \ge 2$. 
If $e=2$, then $A$ is a rational double point. 
In this case, the list of Ulrich ideals of $A$ is given by \cite{GOTWY2}. 
If $e=3$, then $A$ is a rational triple point. 
In this case, the list of Ulrich ideals of $A$ can be seen in the previous section. 
So it is enough to consider the case where $e \ge 4$. 
The following theorem gives an answer in this case. 

%%%%%  Theorem 4.1
\begin{thm} \label{UQuotMult4}
Let $(A,\m)$ be a two-dimensional quotient singularity. 
If $e_0(A) \ge 4$, then $A$ has the unique Ulrich ideal $\m$, that is, 
${\mathcal X}_A=\{\m\}$. 
\end{thm}

\par \vspace{2mm}
In what follows, let $A$ be a two-dimensional quotient singularity. 
The weighted dual graph of $A$ is as follows: 

\begin{enumerate}

\item Cyclic quotient singularity ($n \ge 1$)

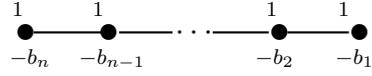
\begin{picture}(200,60)(-40,0)
    \thicklines
\put(25,28){{\tiny $1$}}
\put(30,22){\circle*{6}}
\put(33,22){\line(1,0){30}}
\put(24,10){{\tiny $-b_n$}}
\put(55,28){{\tiny $1$}}
\put(61,22){\circle*{6}}
\put(50,10){{\tiny $-b_{n-1}$}}
\put(65,22){\line(1,0){20}}
\put(86,19){$\cdots$}
\put(120,28){{\tiny $1$}}
\put(100,22){\line(1,0){21}}
\put(125,22){\circle*{6}}
\put(116,10){{\tiny $-b_2$}}
\put(147,28){{\tiny $1$}}
\put(128,22){\line(1,0){23}}
\put(155,22){\circle*{6}}
\put(146,10){{\tiny $-b_1$}}
\end{picture}

\item Type (2,2,n)  $(n \ge 1)$

\begin{picture}(200,60)(-40,0)
    \thicklines
%\put(25,23){{\tiny $1$}}
\put(30,17){\circle*{6}}
\put(33,17){\line(1,0){30}}
\put(24,5){{\tiny $-b_n$}}
%
%\put(55,23){{\tiny $1$}}
\put(61,17){\circle*{6}}
\put(50,5){{\tiny $-b_{n-1}$}}
\put(65,17){\line(1,0){20}}
\put(86,14){$\cdots$}
%
%\put(120,23){{\tiny $1$}}
\put(100,17){\line(1,0){21}}
\put(125,17){\circle*{6}}
\put(116,5){{\tiny $-b_1$}}
%
%\put(147,23){{\tiny $1$}}
\put(128,17){\line(1,0){23}}
\put(155,17){\circle*{6}}
\put(148,5){{\tiny $-b$}}
%
%\put(177,23){{\tiny $1$}}
\put(158,17){\line(1,0){24}}
\put(185,17){\circle{6}}
%
%\put(146,43){{\tiny $1$}}
\put(155,19){\line(0,1){21}}
\put(155,43){\circle{6}}
\end{picture}

%%%%%%%%%%%%%%%%%%%%%%%%%%%%%%%%%%%%%%%%%%%%%%%%%%
%%%%%%%%%%%%%%%%%%%%%%%%%%%%%%%%%%%%%%%%%%%%%%%%%%
\item Type (2,3,3)

\begin{picture}(200,60)(-30,0)
    \thicklines
\put(-20,45){$\Gamma_1(b)$}
%\put(20,23){{\tiny $1$}}
\put(25,17){\circle*{6}}
\put(16,5){{\tiny $-3$}}
%
%\put(47,23){{\tiny $1$}}
\put(57,23){{$E_0$}}
\put(28,17){\line(1,0){23}}
\put(55,17){\circle*{6}}
\put(48,5){{\tiny $-b$}}
%
%\put(77,23){{\tiny $1$}}
\put(58,17){\line(1,0){24}}
\put(85,17){\circle*{6}}
\put(78,5){{\tiny $-3$}}
\put(82,23){{$F$}}
%
%\put(46,43){{\tiny $1$}}
\put(55,19){\line(0,1){21}}
\put(55,43){\circle{6}}
%\end{picture}

%\begin{picture}(200,60)(-40,0)
%    \thicklines
%%%%%%%%%%%%%%%%%%%%%%%%%%%%%%%%%%%%%%%%%
\put(170,45){$\Gamma_2(b)$}
%\put(188,23){{\tiny $1$}}
\put(197,17){\line(1,0){24}}
\put(193,17){\circle{6}}
%\put(183,5){{\tiny $$}}

%
%\put(220,23){{\tiny $1$}}
\put(225,17){\circle{6}}
%\put(216,5){{\tiny $-3$}}
%
%\put(247,23){{\tiny $1$}}
\put(257,23){{$E_0$}}
\put(228,17){\line(1,0){23}}
\put(255,17){\circle*{6}}
\put(248,5){{\tiny $-b$}}
%
%\put(277,23){{\tiny $1$}}
\put(258,17){\line(1,0){24}}
\put(285,17){\circle*{6}}
\put(278,5){{\tiny $-3$}}
\put(282,23){{$F$}}
%
%\put(246,43){{\tiny $1$}}
\put(255,19){\line(0,1){21}}
\put(255,43){\circle{6}}
\end{picture}

%%%%%%%%%%%%%%%%%%%%%%%%%%%%%%%%%%%%%%
\begin{picture}(200,60)(0,0)
\thicklines
\put(10,45){$\Gamma_3(b)$}
%\put(38,23){{\tiny $1$}}
\put(27,17){\line(1,0){24}}
\put(23,17){\circle{6}}
%\put(13,5){{\tiny $$}}
%
%\put(50,23){{\tiny $1$}}
\put(55,17){\circle{6}}
%\put(46,5){{\tiny $-3$}}
%
%\put(77,23){{\tiny $1$}}
\put(87,23){{$E_0$}}
\put(58,17){\line(1,0){23}}
\put(85,17){\circle*{6}}
\put(78,5){{\tiny $-b$}}
%
%\put(27,23){{\tiny $1$}}
\put(88,17){\line(1,0){24}}
\put(115,17){\circle{6}}
%\put(108,5){{\tiny $-3$}}
%
%\put(76,43){{\tiny $1$}}
\put(85,19){\line(0,1){21}}
\put(85,43){\circle{6}}
%
%\put(137,23){{\tiny $1$}}
\put(118,17){\line(1,0){24}}
\put(145,17){\circle{6}}
%\put(138,5){{\tiny $-3$}}
\end{picture}

%%%%%%%%%%%%%%%%%%%%%%%%%%%%%%%%%%%%%%%%
\item Type (2,3,4)

\begin{picture}(200,60)(0,0)
    \thicklines
\put(10,45){$\Gamma_4(b)$}
%\put(50,23){{\tiny $1$}}
\put(55,17){\circle*{6}}
\put(46,5){{\tiny $-3$}}
%
%\put(77,23){{\tiny $1$}}
\put(87,23){{$E_0$}}
\put(58,17){\line(1,0){23}}
\put(85,17){\circle*{6}}
\put(78,5){{\tiny $-b$}}
%
%\put(107,23){{\tiny $1$}}
\put(88,17){\line(1,0){24}}
\put(115,17){\circle*{6}}
\put(108,5){{\tiny $-4$}}
\put(112,23){{$F$}}
%
%\put(76,43){{\tiny $1$}}
\put(85,19){\line(0,1){21}}
\put(85,43){\circle{6}}
%\end{picture}

%\begin{picture}(200,60)(-40,0)
%    \thicklines
%%%%%%%%%%%%%%%%%%%%%%%%%%%%%%%%%%%%%%%%%%
%\put(188,23){{\tiny $1$}}
\put(200,45){$\Gamma_5(b)$}
\put(197,17){\line(1,0){24}}
\put(193,17){\circle{6}}
%\put(183,5){{\tiny $$}}
%
%\put(218,23){{\tiny $1$}}
\put(227,17){\line(1,0){24}}
\put(223,17){\circle{6}}
%\put(213,5){{\tiny $$}}
%
%\put(250,23){{\tiny $1$}}
\put(255,17){\circle{6}}
%\put(246,5){{\tiny $-3$}}
%
%\put(277,23){{\tiny $1$}}
\put(287,23){{$E_0$}}
\put(258,17){\line(1,0){23}}
\put(285,17){\circle*{6}}
\put(278,5){{\tiny $-b$}}
%
%\put(307,23){{\tiny $1$}}
\put(288,17){\line(1,0){24}}
\put(315,17){\circle*{6}}
\put(308,5){{\tiny $-3$}}
\put(312,23){{$F$}}
%
%\put(276,43){{\tiny $1$}}
\put(285,19){\line(0,1){21}}
\put(285,43){\circle{6}}
\end{picture}

%%%%%%%%%%%%%%%%%%%%%%%%%%%%%%%%%%%%%%%
\begin{picture}(200,60)(0,0)
\thicklines
\put(10,45){$\Gamma_6(b)$}
%\put(38,23){{\tiny $1$}}
\put(27,17){\line(1,0){24}}
\put(23,17){\circle{6}}
%\put(13,5){{\tiny $$}}
%
%\put(50,23){{\tiny $1$}}
\put(55,17){\circle{6}}
%\put(46,5){{\tiny $-3$}}
%
%\put(77,23){{\tiny $1$}}
\put(87,23){{$E_0$}}
\put(58,17){\line(1,0){23}}
\put(85,17){\circle*{6}}
\put(78,5){{\tiny $-b$}}
%
%\put(77,23){{\tiny $1$}}
\put(88,17){\line(1,0){24}}
\put(115,17){\circle*{6}}
\put(108,5){{\tiny $-4$}}
\put(113,23){{$F$}}
%
%\put(76,43){{\tiny $1$}}
\put(85,19){\line(0,1){21}}
\put(85,43){\circle{6}}
%
%%%%%%%%%%%%%%%%%%%%%%%%%%%%%%%%%%%%%%
\put(200,45){$\Gamma_7(b)$}
%\put(188,23){{\tiny $1$}}
\put(197,17){\line(1,0){24}}
\put(193,17){\circle{6}}
%\put(183,5){{\tiny $$}}
%
%\put(218,23){{\tiny $1$}}
\put(227,17){\line(1,0){24}}
\put(223,17){\circle{6}}
%\put(213,5){{\tiny $$}}
%
%\put(250,23){{\tiny $1$}}
\put(255,17){\circle{6}}
%\put(246,5){{\tiny $-3$}}
%
%\put(277,23){{\tiny $1$}}
\put(287,23){{$E_0$}}
\put(258,17){\line(1,0){23}}
\put(285,17){\circle*{6}}
\put(278,5){{\tiny $-b$}}
%
%\put(307,23){{\tiny $1$}}
\put(288,17){\line(1,0){24}}
\put(315,17){\circle{6}}
%\put(308,5){{\tiny $-3$}}
%
%\put(276,43){{\tiny $1$}}
\put(285,19){\line(0,1){21}}
\put(285,43){\circle{6}}
%
%\put(337,23){{\tiny $1$}}
\put(318,17){\line(1,0){24}}
\put(345,17){\circle{6}}
%\put(338,5){{\tiny $-3$}}
\end{picture}

%%%%%%%%%%%%%%%%%%%%%%%%%%%%%%%%%%%%%%%%%%%%%
\item Type(2,3,5)

\begin{picture}(200,60)(-30,0)
    \thicklines
\put(-20,45){$\Gamma_8(b)$}
%\put(20,23){{\tiny $1$}}
\put(25,17){\circle*{6}}
\put(16,5){{\tiny $-3$}}
%
%\put(47,23){{\tiny $1$}}
\put(57,23){{$E_0$}}
\put(28,17){\line(1,0){23}}
\put(55,17){\circle*{6}}
\put(48,5){{\tiny $-b$}}
%
%\put(77,23){{\tiny $1$}}
\put(58,17){\line(1,0){24}}
\put(85,17){\circle*{6}}
\put(78,5){{\tiny $-5$}}
\put(82,23){{$F$}}
%
%\put(46,43){{\tiny $1$}}
\put(55,19){\line(0,1){21}}
\put(55,43){\circle{6}}
%\end{picture}

%\begin{picture}(200,60)(-40,0)
%    \thicklines
%%%%%%%%%%%%%%%%%%%%%%%%%%%%%%%%%%%%%%%%%%
\put(170,45){$\Gamma_9(b)$}
%\put(188,23){{\tiny $1$}}
\put(197,17){\line(1,0){24}}
\put(193,17){\circle{6}}
%\put(183,5){{\tiny $$}}

%
%\put(220,23){{\tiny $1$}}
\put(225,17){\circle{6}}
%\put(216,5){{\tiny $-3$}}
%
%\put(247,23){{\tiny $1$}}
\put(257,23){{$E_0$}}
\put(228,17){\line(1,0){23}}
\put(255,17){\circle*{6}}
\put(248,5){{\tiny $-b$}}
%
%\put(277,23){{\tiny $1$}}
\put(258,17){\line(1,0){24}}
\put(285,17){\circle*{6}}
\put(278,5){{\tiny $-5$}}
\put(282,23){{$F$}}
%
%\put(246,43){{\tiny $1$}}
\put(255,19){\line(0,1){21}}
\put(255,43){\circle{6}}
\end{picture}

%%%%%%%%%%%%%%%%%%%%%%%%%%%%%%%%%%%%%%%
\begin{picture}(200,60)(0,0)
\thicklines
\put(10,45){$\Gamma_{10}(b)$}
%\put(38,23){{\tiny $1$}}
\put(27,17){\line(1,0){24}}
\put(23,17){\circle{6}}
%\put(13,5){{\tiny $$}}
%
%\put(50,23){{\tiny $1$}}
\put(55,17){\circle*{6}}
\put(46,5){{\tiny $-3$}}
%
%\put(77,23){{\tiny $1$}}
\put(87,23){{$E_0$}}
\put(58,17){\line(1,0){23}}
\put(85,17){\circle*{6}}
\put(78,5){{\tiny $-b$}}
%
%\put(27,23){{\tiny $1$}}
\put(88,17){\line(1,0){24}}
\put(115,17){\circle*{6}}
\put(108,5){{\tiny $-3$}}
\put(112,23){{$F$}}
%
%\put(76,43){{\tiny $1$}}
\put(85,19){\line(0,1){21}}
\put(85,43){\circle{6}}
%
%%%%%%%%%%%%%%%%%%%%%%%%%%%%%%%%%%%%%%
%\begin{picture}(200,60)(-40,0)
\put(200,45){$\Gamma_{11}(b)$}
%\put(218,23){{\tiny $1$}}
\put(227,17){\line(1,0){24}}
\put(223,17){\circle{6}}
%\put(213,5){{\tiny $$}}
%
%\put(250,23){{\tiny $1$}}
\put(255,17){\circle{6}}
%\put(246,5){{\tiny $-3$}}
%
%\put(277,23){{\tiny $1$}}
\put(287,23){{$E_0$}}
\put(258,17){\line(1,0){23}}
\put(285,17){\circle*{6}}
\put(278,5){{\tiny $-b$}}
%
%\put(307,23){{\tiny $1$}}
\put(288,17){\line(1,0){24}}
\put(315,17){\circle*{6}}
\put(308,5){{\tiny $-3$}}
\put(312,23){{$F$}}
%
%\put(276,43){{\tiny $1$}}
\put(285,19){\line(0,1){21}}
\put(285,43){\circle{6}}
%
%\put(337,23){{\tiny $1$}}
\put(318,17){\line(1,0){24}}
\put(345,17){\circle{6}}
%\put(338,5){{\tiny $-3$}}
\end{picture}

%%%%%%%%%%%%%%%%%%%%%%%%%%%%%%%%%%%%%%%
\begin{picture}(200,60)(0,0)
\thicklines
\put(10,45){$\Gamma_{12}(b)$}
%\put(38,23){{\tiny $1$}}
\put(27,17){\line(1,0){24}}
\put(23,17){\circle*{6}}
\put(13,5){{\tiny $-3$}}
%
%\put(50,23){{\tiny $1$}}
\put(55,17){\circle{6}}
%\put(46,5){{\tiny $-3$}}
%
%\put(77,23){{\tiny $1$}}
\put(87,23){{$E_0$}}
\put(58,17){\line(1,0){23}}
\put(85,17){\circle*{6}}
\put(78,5){{\tiny $-b$}}
%
%\put(27,23){{\tiny $1$}}
\put(88,17){\line(1,0){24}}
\put(115,17){\circle*{6}}
\put(108,5){{\tiny $-3$}}
\put(112,23){{$F$}}
%
%\put(76,43){{\tiny $1$}}
\put(85,19){\line(0,1){21}}
\put(85,43){\circle{6}}
%
%%%%%%%%%%%%%%%%%%%%%%%%%%%%%%%%%%%%%%
%\begin{picture}(200,60)(-40,0)
\put(200,45){$\Gamma_{13}(b)$}
%\put(218,23){{\tiny $1$}}
\put(227,17){\line(1,0){24}}
\put(223,17){\circle{6}}
%\put(213,5){{\tiny $$}}
%
%\put(250,23){{\tiny $1$}}
\put(255,17){\circle{6}}
%\put(246,5){{\tiny $-3$}}
%
%\put(277,23){{\tiny $1$}}
\put(287,23){{$E_0$}}
\put(258,17){\line(1,0){23}}
\put(285,17){\circle*{6}}
\put(278,5){{\tiny $-b$}}
%
%\put(307,23){{\tiny $1$}}
\put(288,17){\line(1,0){24}}
\put(315,17){\circle{6}}
%\put(308,5){{\tiny $-3$}}
%
%\put(276,43){{\tiny $1$}}
\put(285,19){\line(0,1){21}}
\put(285,43){\circle{6}}
%
%\put(337,23){{\tiny $1$}}
\put(318,17){\line(1,0){24}}
\put(345,17){\circle*{6}}
\put(338,5){{\tiny $-3$}}
\put(342,23){{$F$}}
\end{picture}

%%%%%%%%%%%%%%%%%%%%%%%%%%%%%%%%%%%%%%%
\begin{picture}(200,60)(80,0)
\thicklines
\put(90,45){$\Gamma_{14}(b)$}
%\put(58,23){{\tiny $1$}}
\put(47,17){\line(1,0){24}}
\put(43,17){\circle{6}}
%\put(33,5){{\tiny $$}}
%
%\put(88,23){{\tiny $1$}}
\put(77,17){\line(1,0){24}}
\put(73,17){\circle{6}}
%\put(63,5){{\tiny $$}}
%
%\put(118,23){{\tiny $1$}}
\put(107,17){\line(1,0){24}}
\put(103,17){\circle{6}}
%\put(93,5){{\tiny $$}}
%
%\put(130,23){{\tiny $1$}}
\put(135,17){\circle{6}}
%\put(126,5){{\tiny $-3$}}
%
%\put(157,23){{\tiny $1$}}
\put(167,23){{$E_0$}}
\put(138,17){\line(1,0){23}}
\put(165,17){\circle*{6}}
\put(158,5){{\tiny $-b$}}
%
%\put(107,23){{\tiny $1$}}
\put(168,17){\line(1,0){24}}
\put(195,17){\circle*{6}}
\put(188,5){{\tiny $-3$}}
\put(192,23){{$F$}}
%
%\put(156,43){{\tiny $1$}}
\put(165,19){\line(0,1){21}}
\put(165,43){\circle{6}}
%
%%%%%%%%%%%%%%%%%%%%%%%%%%%%%%%%%%%%%%
\put(280,45){$\Gamma_{15}(b)$}
%\put(238,23){{\tiny $1$}}
\put(247,17){\line(1,0){24}}
\put(243,17){\circle{6}}
%\put(233,5){{\tiny $$}}
%
%\put(268,23){{\tiny $1$}}
\put(277,17){\line(1,0){24}}
\put(273,17){\circle{6}}
%\put(263,5){{\tiny $$}}
%
%\put(298,23){{\tiny $1$}}
\put(307,17){\line(1,0){24}}
\put(303,17){\circle{6}}
%\put(293,5){{\tiny $$}}
%
%\put(330,23){{\tiny $1$}}
\put(335,17){\circle{6}}
%\put(326,5){{\tiny $-3$}}
%
%\put(357,23){{\tiny $1$}}
\put(367,23){{$E_0$}}
\put(338,17){\line(1,0){23}}
\put(365,17){\circle*{6}}
\put(358,5){{\tiny $-b$}}
%
%\put(387,23){{\tiny $1$}}
\put(368,17){\line(1,0){24}}
\put(395,17){\circle{6}}
%\put(388,5){{\tiny $-3$}}
%
%\put(356,43){{\tiny $1$}}
\put(365,19){\line(0,1){21}}
\put(365,43){\circle{6}}
%
%\put(417,23){{\tiny $1$}}
\put(398,17){\line(1,0){24}}
\put(425,17){\circle{6}}
%\put(418,5){{\tiny $-3$}}
\end{picture}

\end{enumerate}
\bigskip
Proof of Theorem \ref{UQuotMult4}. 
\par \vspace{1mm} \par \noindent 
{\bf Case 1. Cyclic Quotient Singularity} 
\par 
By \cite[Theorem A]{CS}, any cyclic quotient singularity (of dimension $2$) is a nearly Gorenstein ring, and thus ${\mathcal X}_A=\{\m\}$. 
Indeed, we can prove ${\mathcal X}_A=\{\m\}$ directly. 
The fundamental cycle is reduced, that is, $Z_0=E_1+E_2+\cdots+E_n$ 
with $E_i^2=-b_i$ $(i=1,2,\ldots,n)$ as follows. 
\par 
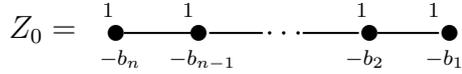
\vspace{2mm}
\begin{picture}(200,60)(-30,0)
    \thicklines
\put(-10,20){$Z_0=$}
\put(25,28){{\tiny $1$}}
\put(30,22){\circle*{6}}
\put(33,22){\line(1,0){30}}
\put(24,10){{\tiny $-b_n$}}
\put(55,28){{\tiny $1$}}
\put(61,22){\circle*{6}}
\put(50,10){{\tiny $-b_{n-1}$}}
\put(65,22){\line(1,0){20}}
\put(86,19){$\cdots$}
\put(120,28){{\tiny $1$}}
\put(100,22){\line(1,0){21}}
\put(125,22){\circle*{6}}
\put(116,10){{\tiny $-b_2$}}
\put(147,28){{\tiny $1$}}
\put(128,22){\line(1,0){23}}
\put(155,22){\circle*{6}}
\put(146,10){{\tiny $-b_1$}}
\end{picture}
\par \vspace{2mm} \par \noindent
As $e_0(A) \ge 3$, we can take an integer $j \ge 1$ with 
$-E_j^2=b_j \ge 3$. 
Then $-E_jZ_0 \ge b_j -2 > 0$. 
Thus Lemma \ref{KeyLemma} yields ${\mathcal X}_A=\{\m\}$.  

\medskip \par \noindent 
{\bf Case 2. Type (2,2,n).} 
Let $E_0$ be the central curve with $-E_0^2=b$. 
\par \vspace{2mm}
When $b \ge 4$, the fundamental cycle is as follows: 

\par 
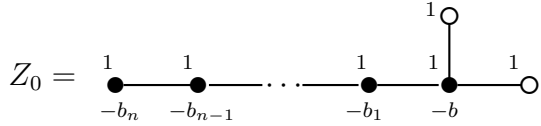
\vspace{2mm}
\begin{picture}(200,60)(-40,0)
    \thicklines
\put(-10,17){$Z_0=$}
\put(25,23){{\tiny $1$}}
\put(30,17){\circle*{6}}
\put(33,17){\line(1,0){30}}
\put(24,5){{\tiny $-b_n$}}
\put(55,23){{\tiny $1$}}
\put(61,17){\circle*{6}}
\put(50,5){{\tiny $-b_{n-1}$}}
\put(65,17){\line(1,0){20}}
\put(86,14){$\cdots$}
\put(120,23){{\tiny $1$}}
\put(100,17){\line(1,0){21}}
\put(125,17){\circle*{6}}
\put(116,5){{\tiny $-b_1$}}
\put(147,23){{\tiny $1$}}
\put(128,17){\line(1,0){23}}
\put(155,17){\circle*{6}}
\put(148,5){{\tiny $-b$}}
\put(177,23){{\tiny $1$}}
\put(158,17){\line(1,0){24}}
\put(185,17){\circle{6}}
\put(146,43){{\tiny $1$}}
\put(155,19){\line(0,1){21}}
\put(155,43){\circle{6}}
\end{picture}
\par \vspace{2mm}\par \noindent 
Since $-E_0Z_0=b-3 > 0$, the assertion follows from 
Lemma \ref{KeyLemma}.  
\par \vspace{2mm}
When $b=3$, $Z_0$ is the same shape as in the above case. 
As $e_0(A) \ge 4$ by assumption, there exists an integer $j \ge 1$
such that $-E_j^2=b_j \ge 3$. 
Then $-E_jZ_0 \ge b_j-2 >0$. 
Hence  the assertion follows from 
Lemma \ref{KeyLemma}.  
\par \vspace{2mm}
When $b=2$, let $j \ge 1$ be an integer 
such that $b_1=\cdots = b_{j-1}=2$ and $b_j \ge 3$.  
Now suppose that $b_j=3$. Then there exists an integer 
$k > j$ such that $b_k \ge 3$. 
Then $Z_0$ is the following shape: 
\par 
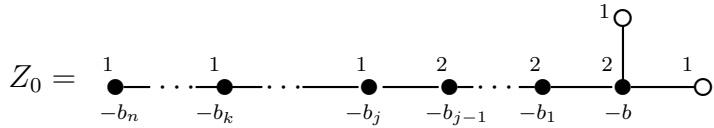
\vspace{2mm}
\begin{picture}(200,60)(-40,0)
    \thicklines
\put(-10,17){$Z_0=$}
\put(25,23){{\tiny $1$}}
\put(30,17){\circle*{6}}
\put(33,17){\line(1,0){10}}
\put(46,14){$\cdots$}
\put(24,5){{\tiny $-b_n$}}
\put(65,23){{\tiny $1$}}
\put(71,17){\circle*{6}}
\put(60,5){{\tiny $-b_k$}}
\put(60,17){\line(1,0){25}}
\put(86,14){$\cdots$}
\put(120,23){{\tiny $1$}}
\put(100,17){\line(1,0){21}}
\put(125,17){\circle*{6}}
\put(116,5){{\tiny $-b_j$}}
\put(150,23){{\tiny $2$}}
\put(130,17){\line(1,0){21}}
\put(155,17){\circle*{6}}
\put(146,5){{\tiny $-b_{j-1}$}}
\put(185,23){{\tiny $2$}}
\put(158,17){\line(1,0){6}}
\put(165,14){$\cdots$}
\put(181,17){\line(1,0){6}}
\put(190,17){\circle*{6}}
\put(181,5){{\tiny $-b_1$}}
\put(212,23){{\tiny $2$}}
\put(193,17){\line(1,0){23}}
\put(220,17){\circle*{6}}
\put(213,5){{\tiny $-b$}}
\put(242,23){{\tiny $1$}}
\put(223,17){\line(1,0){24}}
\put(250,17){\circle{6}}
\put(211,43){{\tiny $1$}}
\put(220,19){\line(0,1){21}}
\put(220,43){\circle{6}}
\end{picture}
\par \vspace{2mm} \par \noindent 
Then $-E_k^2=b_k \ge 3$ and $-E_kZ_0 \ge b_k-2 >0$. 
\par \vspace{1mm}
Otherwise, $b_j \ge 4$. 
Then $-E_j^2=b_j \ge 3$ and 
$-E_jZ_0 =b_j -3 > 0$. 
In both cases, we have ${\mathcal X}_A=\{\m\}$ by Lemma \ref{KeyLemma}.

\medskip \par \noindent 
{\bf Case 3. Type (2,3,3), (2,3,4), (2,3,5).} 
~ \par \vspace{1mm}
Let $E_0$ be the central curve with $-E_0^2=b$. 
\par \vspace{2mm}
When $b \ge 4$, the fundamental cycle $Z_0$ is given by the following shape:

\par 
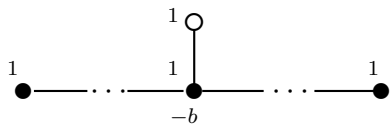
\vspace{2mm}
\begin{picture}(200,60)(-40,0)
    \thicklines
\put(55,23){{\tiny $1$}}
\put(61,17){\circle*{6}}
%\put(50,5){{\tiny $-b_{n-1}$}}
%
\put(65,17){\line(1,0){20}}
\put(86,14){$\cdots$}
\put(115,23){{\tiny $1$}}
\put(100,17){\line(1,0){21}}
\put(125,17){\circle*{6}}
\put(116,5){{\tiny $-b$}}
%
%\put(147,23){{\tiny $1$}}
\put(128,17){\line(1,0){23}}
%\put(155,17){\circle*{6}}
%\put(148,5){{\tiny $-b$}}
\put(154,14){$\cdots$}
\put(190,23){{\tiny $1$}}
\put(170,17){\line(1,0){24}}
\put(195,17){\circle*{6}}
\put(115,43){{\tiny $1$}}
\put(125,19){\line(0,1){21}}
\put(125,43){\circle{6}}
\end{picture}
\par \vspace{2mm} \par \noindent 
Then since $-E_0^2=b \ge 3$ and $-E_0Z_0=b-3 >0$, we obtain 
${\mathcal X}_A=\{\m\}$ by Lemma \ref{KeyLemma}. 
\par \vspace{2mm}
When $b=3$, 
the fundamental cycle $Z_0$ is given by the following shape:

\par 
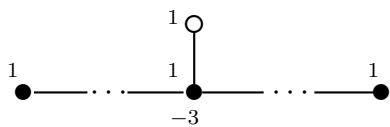
\vspace{2mm}
\begin{picture}(200,60)(-40,0)
    \thicklines
\put(55,23){{\tiny $1$}}
\put(61,17){\circle*{6}}
%\put(50,5){{\tiny $-b_{n-1}$}}
%
\put(65,17){\line(1,0){20}}
\put(86,14){$\cdots$}
\put(115,23){{\tiny $1$}}
\put(100,17){\line(1,0){21}}
\put(125,17){\circle*{6}}
\put(116,5){{\tiny $-3$}}
%
%\put(147,23){{\tiny $1$}}
\put(128,17){\line(1,0){23}}
%\put(155,17){\circle*{6}}
%\put(148,5){{\tiny $-b$}}
\put(154,14){$\cdots$}
\put(190,23){{\tiny $1$}}
\put(170,17){\line(1,0){24}}
\put(195,17){\circle*{6}}
\put(115,43){{\tiny $1$}}
\put(125,19){\line(0,1){21}}
\put(125,43){\circle{6}}
\end{picture}
\par \vspace{2mm}\par \noindent
Since $e_0(A) \ge 4$, it is enough to consider the graph 
$\Gamma_i(3)$ for $i \ne 3,7,15$. 
Let $F$ be the curve which appears in the table. 
If $i=11$, then $-F^2 =3$ and $-FZ_0=1 >0$. 
Otherwise, $-F^2 \ge 3$ and $-FZ_0=-F^2-1 >0$. 
Hence we obtain ${\mathcal X}_A=\{\m\}$ by Lemma \ref{KeyLemma}. 
\par \vspace{2mm}
When $b=2$, it is enough to consider the graph $\Gamma_i(2)$ for 
$i=1,4,6,8,9,10,12$ only.   

%%%%%%%%%%%%%%%%%%%%%%%%%%%%%%%
\par 
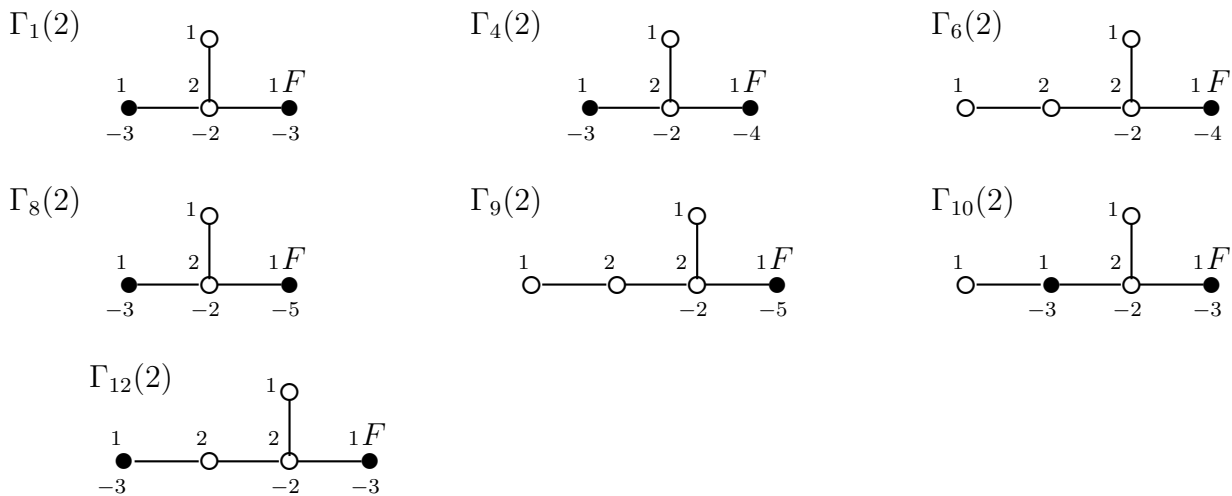
\vspace{3mm}
\begin{picture}(200,60)(0,0)
    \thicklines
\put(-20,45){$\Gamma_1(2)$}
\put(20,23){{\tiny $1$}}
\put(25,17){\circle*{6}}
\put(16,5){{\tiny $-3$}}
\put(47,23){{\tiny $2$}}
\put(28,17){\line(1,0){23}}
\put(55,17){\circle{6}}
\put(48,5){{\tiny $-2$}}
\put(77,23){{\tiny $1$}}
\put(58,17){\line(1,0){24}}
\put(85,17){\circle*{6}}
\put(78,5){{\tiny $-3$}}
\put(82,23){{$F$}}
\put(46,43){{\tiny $1$}}
\put(55,19){\line(0,1){21}}
\put(55,43){\circle{6}}
\end{picture}
%%%%%%%%%%%%%%%%%%%%%%%%%%%%%
\begin{picture}(200,60)(60,0)
    \thicklines
\put(10,45){$\Gamma_4(2)$}
\put(50,23){{\tiny $1$}}
\put(55,17){\circle*{6}}
\put(46,5){{\tiny $-3$}}
\put(77,23){{\tiny $2$}}
\put(58,17){\line(1,0){23}}
\put(85,17){\circle{6}}
\put(78,5){{\tiny $-2$}}
\put(107,23){{\tiny $1$}}
\put(88,17){\line(1,0){24}}
\put(115,17){\circle*{6}}
\put(108,5){{\tiny $-4$}}
\put(112,23){{$F$}}
\put(76,43){{\tiny $1$}}
\put(85,19){\line(0,1){21}}
\put(85,43){\circle{6}}
\end{picture}
%%%%%%%%%%%%%%%%%%%%%%%%%%%%%%%
\begin{picture}(200,60)(90,0)
\thicklines
\put(10,45){$\Gamma_6(2)$}
\put(18,23){{\tiny $1$}}
\put(27,17){\line(1,0){24}}
\put(23,17){\circle{6}}

\put(50,23){{\tiny $2$}}
\put(55,17){\circle{6}}

\put(77,23){{\tiny $2$}}
\put(58,17){\line(1,0){23}}
\put(85,17){\circle{6}}
\put(78,5){{\tiny $-2$}}
\put(107,23){{\tiny $1$}}
\put(88,17){\line(1,0){24}}
\put(115,17){\circle*{6}}
\put(108,5){{\tiny $-4$}}
\put(113,23){{$F$}}
\put(76,43){{\tiny $1$}}
\put(85,19){\line(0,1){21}}
\put(85,43){\circle{6}}
\end{picture}
\par \vspace{2mm}
%%%%%%%%%%%%%%%%%%%%%%%%%%%%%%
\begin{picture}(200,60)(0,0)
    \thicklines
\put(-20,45){$\Gamma_8(2)$}
\put(20,23){{\tiny $1$}}
\put(25,17){\circle*{6}}
\put(16,5){{\tiny $-3$}}
\put(47,23){{\tiny $2$}}
\put(28,17){\line(1,0){23}}
\put(55,17){\circle{6}}
\put(48,5){{\tiny $-2$}}
\put(77,23){{\tiny $1$}}
\put(58,17){\line(1,0){24}}
\put(85,17){\circle*{6}}
\put(78,5){{\tiny $-5$}}
\put(82,23){{$F$}}
\put(46,43){{\tiny $1$}}
\put(55,19){\line(0,1){21}}
\put(55,43){\circle{6}}
\end{picture}
%%%%%%%%%%%%%
\begin{picture}(200,60)(120,0)
    \thicklines
\put(70,45){$\Gamma_9(2)$}
\put(88,23){{\tiny $1$}}
\put(97,17){\line(1,0){24}}
\put(93,17){\circle{6}}
%\put(183,5){{\tiny $$}}

%
\put(120,23){{\tiny $2$}}
\put(125,17){\circle{6}}
%\put(116,5){{\tiny $-3$}}
%
\put(147,23){{\tiny $2$}}
\put(128,17){\line(1,0){23}}
\put(155,17){\circle{6}}
\put(148,5){{\tiny $-2$}}
\put(177,23){{\tiny $1$}}
\put(158,17){\line(1,0){24}}
\put(185,17){\circle*{6}}
\put(178,5){{\tiny $-5$}}
\put(182,23){{$F$}}
\put(146,43){{\tiny $1$}}
\put(155,19){\line(0,1){21}}
\put(155,43){\circle{6}}
\end{picture}
%%%%%%%%%%%%%%%%%%%%%%%%%%%%%%%%%%%%%%%
\begin{picture}(200,60)(90,0)
\thicklines
\put(10,45){$\Gamma_{10}(2)$}
\put(18,23){{\tiny $1$}}
\put(27,17){\line(1,0){24}}
\put(23,17){\circle{6}}
%\put(13,5){{\tiny $$}}
%
\put(50,23){{\tiny $1$}}
\put(55,17){\circle*{6}}
\put(46,5){{\tiny $-3$}}
\put(77,23){{\tiny $2$}}
\put(58,17){\line(1,0){23}}
\put(85,17){\circle{6}}
\put(78,5){{\tiny $-2$}}
\put(108,23){{\tiny $1$}}
\put(88,17){\line(1,0){24}}
\put(115,17){\circle*{6}}
\put(108,5){{\tiny $-3$}}
\put(113,23){{$F$}}
\put(76,43){{\tiny $1$}}
\put(85,19){\line(0,1){21}}
\put(85,43){\circle{6}}
\end{picture}
\par \vspace{2mm}
%%%%%%%%%%%%%%%%%%%%%%%%%%%%%%%%%%%%%%%
\begin{picture}(200,60)(0,0)
\thicklines
\put(10,45){$\Gamma_{12}(2)$}
\put(18,23){{\tiny $1$}}
\put(27,17){\line(1,0){24}}
\put(23,17){\circle*{6}}
\put(13,5){{\tiny $-3$}}
\put(50,23){{\tiny $2$}}
\put(55,17){\circle{6}}
%\put(46,5){{\tiny $-3$}}
%
\put(77,23){{\tiny $2$}}
\put(58,17){\line(1,0){23}}
\put(85,17){\circle{6}}
\put(78,5){{\tiny $-2$}}
\put(107,23){{\tiny $1$}}
\put(88,17){\line(1,0){24}}
\put(115,17){\circle*{6}}
\put(108,5){{\tiny $-3$}}
\put(112,23){{$F$}}
\put(76,43){{\tiny $1$}}
\put(85,19){\line(0,1){21}}
\put(85,43){\circle{6}}
\end{picture}
\par \vspace{2mm}
In each graph, we have $-F^2 \ge 3$ and $-FZ_0 = -F^2-2 > 0$. 
Hence ${\mathcal X}_A=\{\m\}$ by Lemma \ref{KeyLemma}.

\par \vspace{5mm}
%%%%  Corollary
\begin{cor} \label{QuotUlrich}
Any two-dimensional quotient singularity
$A$ has at most two Ulrich ideals. 
Moreover, $\sharp(\mathcal{X}_A)\le 2$. 
\end{cor}

\begin{proof}
By Theorem \ref{UQuotMult4}, we may assume $e_0(A) =3$.  
Moreover, we may assume that $A$ is not nearly Gorenstein. 
\par 
Let $H$ be the weighted dual graph of the minimal resolution.  
If $H$ is type $(2,2,n)$, then $H=A_{1,1,n}$ $(n \ge 1)$, 
$B_{n-1,3}$ $(n \ge 1)$ or $C_{n+2-k,k}$ $(k \ge 4, n \ge k-1)$. 
Then Corollary \ref{RTP-NG} implies $\sharp({\mathcal X}_A) =2$. 
\par 
If $H$ has type $(2,3,3)$, then $H=\Gamma_3(3)$, 
that is, $H=A_{1,2,2}$. 
Then Corollary \ref{RTP-NG} implies $\sharp({\mathcal X}_A) =2$. 
\par 
If $H$ has type $(2,3,4)$, then $H=\Gamma_7(3)$, 
that is, $H=A_{1,2,3}$. 
Then Corollary \ref{RTP-NG} implies $\sharp({\mathcal X}_A) =2$. 
\par 
If $H$ has type $(2,3,5)$, then $H=\Gamma_{11}(2)$ or 
$\Gamma_{15}(3)$, that is, $H=B_{1,4}$ or $A_{1,2,4}$. 
Then Corollary \ref{RTP-NG} implies $\sharp({\mathcal X}_A) =2$. 
\end{proof}

%%%%%%%%%%%%%%%%%%%%%%%%%%%%%%%%%%%%%%%%%%%%%%%%%
%%%%%%%%%%%%%%%%%%%%%%%%%%%%%%%%%%%%%%%%%%%%%%%%%
%%%%%  Section 5
\section{Examples}

In this section, we give several examples.  
The next example is not only a rational triple point but also 
a quotient singularity. 

\begin{exam} \label{Sample}
Suppose $H=A_{1,2,3}$ and 
let $A$ be the $(x,y,z,t)$-adic completion of $R(H)$, where 
\begin{equation*}
\begin{aligned}
R(H)&= k[t,x,y,z]/(xy-t^{5},xz-t^{6}-zt^2, yz+yt^{4}-zt^{3})\\ 
&= k[t,x,y,z]/I_2
\begin{pmatrix}
x & t^3& t^4+z \\[2mm]
t^2 & y & z \\
\end{pmatrix}_.
\end{aligned}
\end{equation*}
Then $A$ has two Ulrich ideals: 
$I_0=\m=(x,y,z,t)=I_{Z_0}$ and 
$I_1=\tr_A(\omega_A)=(x,y,z,t^2)=I_{Z_1}$, where 
$Z_0,Z_1$ as follows. 

\par 
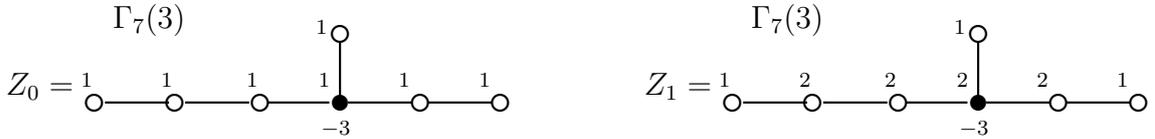
\vspace{2mm}
\begin{picture}(200,60)(65,0)
\thicklines
\put(60,20){$Z_0=$}
\put(100,45){$\Gamma_7(3)$}
\put(88,23){{\tiny $1$}}
\put(97,17){\line(1,0){24}}
\put(93,17){\circle{6}}
%\put(183,5){{\tiny $$}}
%
\put(118,23){{\tiny $1$}}
\put(127,17){\line(1,0){24}}
\put(123,17){\circle{6}}
%\put(113,5){{\tiny $$}}
%
\put(150,23){{\tiny $1$}}
\put(155,17){\circle{6}}
%\put(146,5){{\tiny $-3$}}
%
\put(177,23){{\tiny $1$}}
%\put(187,23){{$E_0$}}
\put(158,17){\line(1,0){23}}
\put(185,17){\circle*{6}}
\put(178,5){{\tiny $-3$}}
\put(207,23){{\tiny $1$}}
\put(188,17){\line(1,0){24}}
\put(215,17){\circle{6}}
%\put(208,5){{\tiny $-3$}}
%
\put(176,43){{\tiny $1$}}
\put(185,19){\line(0,1){21}}
\put(185,43){\circle{6}}
\put(237,23){{\tiny $1$}}
\put(218,17){\line(1,0){24}}
\put(245,17){\circle{6}}
%\put(238,5){{\tiny $-3$}}
\end{picture}
\begin{picture}(200,60)(30,0)
\thicklines
\put(60,20){$Z_1=$}
\put(100,45){$\Gamma_7(3)$}
\put(88,23){{\tiny $1$}}
\put(97,17){\line(1,0){24}}
\put(93,17){\circle{6}}
%\put(183,5){{\tiny $$}}
%
\put(118,23){{\tiny $2$}}
\put(127,17){\line(1,0){24}}
\put(123,17){\circle{6}}
%\put(113,5){{\tiny $$}}
%
\put(150,23){{\tiny $2$}}
\put(155,17){\circle{6}}
%\put(146,5){{\tiny $-3$}}
%
\put(177,23){{\tiny $2$}}
%\put(187,23){{$E_0$}}
\put(158,17){\line(1,0){23}}
\put(185,17){\circle*{6}}
\put(178,5){{\tiny $-3$}}
\put(207,23){{\tiny $2$}}
\put(188,17){\line(1,0){24}}
\put(215,17){\circle{6}}
%\put(208,5){{\tiny $-3$}}
%
\put(176,43){{\tiny $1$}}
\put(185,19){\line(0,1){21}}
\put(185,43){\circle{6}}
\put(237,23){{\tiny $1$}}
\put(218,17){\line(1,0){24}}
\put(245,17){\circle{6}}
%\put(238,5){{\tiny $-3$}}
\end{picture}
\end{exam}

\par \vspace{3mm}
The following normal local domain is not rational. 
But a similar result as in Theorem \ref{Main-RTP} holds true. 

\begin{exam}[\textrm{cf. \cite[Example 7.5]{OWY2}}] 
\label{Non-Rat-Ex}
Let $A=\mathbb{C}[[x,y,z,t]]/I_2(\mathbb{M})$, where 
\[
\mathbb{M}=\left(
\begin{array}{ccc}
x & y & z \\
y & z & x^2-t^3
\end{array}
\right). 
\]
Then $A$ is a two-dimensional normal local domain with $e_0(A)=3$. 
\begin{enumerate}
\item $A$ is \textit{not} a rational singularity. Indeed, $p_g(A)=1$. 
\item $\tr(\omega_A)=(x,y,z,t^3)$ by Lemma \ref{Burch}. 
\item ${\mathcal X}_A=\{(x,y,z,t^i) \,|\, i=1,2,3\}$. 
In fact, if we put $J_i=(x,y,z,t^i), Q_i=(x,t^i)$ for each $i=1,2,3$, 
then $J_i^2=Q_iJ_i$ and $e_0(J_i)=3i=(\mu(J_i)-1)\ell(A/J_i)$. 
Hence it follows from Lemma \ref{CriterionUlrich} that 
$J_i$ is an Ulrich ideal for every $i=1,2,3$. 
\end{enumerate}
\end{exam}

\par \vspace{3mm}
Using the criterion given by T.~Okuma and K.-i.~Watanabe, 
we can show that the following example $A$ is not nearly Gorenstein. 
Hence  
\lq\lq ${\mathcal X}_A=\{\m\}$'' does not necessarily  imply that \lq\lq $A$ is nearly Gorenstein'' when $e_0(A) \ge 4$.  

\begin{exam}[\textrm{cf. \cite{Ri}}] \label{Non-NG}
Let $A=\mathbb{C}[[z_1,z_2,z_3,z_4,z_5]]/I_2(\mathbb{M})$, where 
\[
\mathbb{M}=
\left(
\begin{array}{cccc}
z_1 & z_4 & z_2 & z_3^2 \\
z_2 & z_3 & z_4 & z_5 
\end{array}
\right). 
\]
Then $A$ is a two-dimensional quotient singularity with $e_0(A)=4$ 
and ${\mathcal X}_A=\{\m\}$ but \textit{not} nearly Gorenstein. 
\par
Let $Z$ be the anti-nef cycle corresponding to $J=\tr_A(\omega_A)$. 
Then one can easily see that $\ell(A/J)=2$, $e_0(J)=7$ and $\mu(J)=5$. In particular, $J$ is \textit{not} an Ulrich ideal. 

\par 
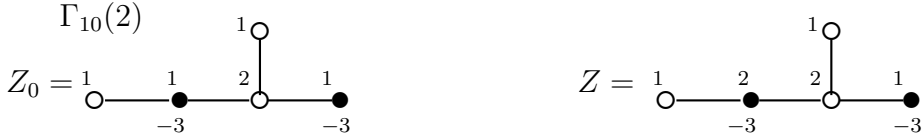
\vspace{3mm}
\begin{picture}(200,60)(-20,0)
\thicklines
\put(-10,20){$Z_0=$}
\put(10,45){$\Gamma_{10}(2)$}
\put(18,23){{\tiny $1$}}
\put(27,17){\line(1,0){24}}
\put(23,17){\circle{6}}
%\put(13,5){{\tiny $$}}
%
\put(50,23){{\tiny $1$}}
\put(55,17){\circle*{6}}
\put(46,5){{\tiny $-3$}}
\put(77,23){{\tiny $2$}}
\put(58,17){\line(1,0){23}}
\put(85,17){\circle{6}}
%\put(78,5){{\tiny $-2$}}
%
\put(108,23){{\tiny $1$}}
\put(88,17){\line(1,0){24}}
\put(115,17){\circle*{6}}
\put(108,5){{\tiny $-3$}}
%\put(113,23){{$F$}}
%
\put(76,43){{\tiny $1$}}
\put(85,19){\line(0,1){21}}
\put(85,43){\circle{6}}
\end{picture}
\begin{picture}(200,60)(-30,0)
\thicklines
\put(-10,20){$Z=$}
%\put(10,45){$\Gamma_{10}(2)$}
\put(18,23){{\tiny $1$}}
\put(27,17){\line(1,0){24}}
\put(23,17){\circle{6}}
%\put(13,5){{\tiny $$}}
%
\put(50,23){{\tiny $2$}}
\put(55,17){\circle*{6}}
\put(46,5){{\tiny $-3$}}
\put(77,23){{\tiny $2$}}
\put(58,17){\line(1,0){23}}
\put(85,17){\circle{6}}
%\put(78,5){{\tiny $-2$}}
%
\put(108,23){{\tiny $1$}}
\put(88,17){\line(1,0){24}}
\put(115,17){\circle*{6}}
\put(108,5){{\tiny $-3$}}
%\put(113,23){{$F$}}
%
\put(76,43){{\tiny $1$}}
\put(85,19){\line(0,1){21}}
\put(85,43){\circle{6}}
\end{picture}
\end{exam}

\par \vspace{2mm}
Let $A$ be a two-dimensional non-Gorenstein rational singularity. 
Then for any Ulrich ideal $I$ in $A$, one has that $A/I$ is Gorenstein;  see \cite[Corollary 6.5]{GOTWY2}. 
So it is natural to ask the following question.

%%%  Question 
\begin{quest}
Let $A$ be a two-dimensional non-Gorenstein rational singularity. 
Then is $A/\tr(\omega_A)$ Gorenstein?  
\end{quest}

\begin{acknowledgement}
The authors thank Tomohiro Okuma, Kei-ichi Watanabe 
and Masataka 
Tomari for many valuable comments. 
\end{acknowledgement}

%%%%%%%%%%%%%%%%%%%%%%%%%%%%%%%%%%%%%%%%%%%%%%%%%%%%%%%%%%
%%%%%%%%%%%%%%%%%%%%%%%%%%%%%%%%%%%%%%%%%%%%%%%%%%%%%%%%%%

%\providecommand{\bysame}{\leavevmode\hbox to3em{\hrulefill}\thinspace}
%\providecommand{\MR}{\relax\ifhmode\unskip\space\fi MR }
%\providecommand{\MRhref}[2]{%
%  \href{http://www.ams.org/mathscinet-getitem?mr=#1}{#2}
%}
%\providecommand{\href}[2]{#2}

\end{document}